\documentclass[a4paper,reqno,11pt]{amsart}

\usepackage{amsmath}
\usepackage{amssymb}
\usepackage{graphicx}
\usepackage{longtable}

\usepackage[all]{xy}% Para los diagramas.
\input{diagxy}  %\input xy\usepackage{xypic} %\xyoption{curve} \xyoption{line}\xyoption{matrix}

\def\pd#1#2{\frac{\partial#1}{\partial#2}}
\def\pois#1#2{\{#1,#2\}}                           %Poisson {f,g}
\def\set#1{\{\,#1\,\}}                 %  notación para conjuntos
\def\>#1{{\bf #1}}                      %  notación para vectores

\def\matrix#1#2{\left(\begin{array}{#1} #2 \end{array}\right)}
\def\eq#1{{\begin{equation} #1 \end{equation}}}
\def\poi{\{\:,\}}
\def\mitad{\frac{1}{2}}

\def\a{\alpha}
\def\w{\omega}                            %  una forma symplectic
\def\R{{{{\mathbb{R}}}}}                 %%real numbers (Pepin)
\def\N{{{{\mathbb{N}}}}}           %%complex numbers (Pepin)

\def\dim{\hbox{{\rm dim}}}                   %%several definitions

\def\rank{\hbox{{\rm rank}}}
\def\feed{\mathrm{Feed}}
\def\feedtot{\mathrm{Feedtot}}
\def\nofeed{\mathrm{Nofeed}}
\def\rfeed{\mathrm{Rfeed}}
\def\pois{\mathrm{POI}}

\def\com #1{\small\% \emph{{#1}}}

\newtheorem{theorem}{Theorem}

\newcommand{\nc}{\newcommand}                        %PARA CAMBIAR COMANDOS
\nc{\rnc}{\renewcommand}

\rnc{\a}{\alpha} \rnc{\b}{\beta} \rnc{\d}{\delta} \nc{\D}{\Delta} \nc{\e}{\varepsilon}
\nc{\g}{\gamma} \nc{\G}{\Gamma} \rnc{\l}{\lambda} \rnc{\L}{\Lambda}  \nc{\n}{\nabla}
\nc{\var}{\varphi} \nc{\ro}{\rho} \nc{\s}{\sigma} \nc{\Sig}{\Sigma} \rnc{\t}{\tau}
\nc{\te}{\theta} \rnc{\o}{\omega} \rnc{\O}{\Omega} \nc{\z}{\zeta}
\rnc{\P}{\Phi} \nc{\Te}{\Theta}
\nc{\p}{\partial}

\def\eqy#1{{\begin{eqnarray} #1 \end{eqnarray}}}
\def\eqi#1{{\begin{eqnarray*} #1 \end{eqnarray*}}}
\def\matrix#1#2{\left[\begin{array}{#1} #2 \end{array}\right]}

\begin{document}
\date{\bf \today}

\title[A Hamiltonian Algorithm for Singular LQ  Systems]{A Hamiltonian  Algorithm for Singular Optimal LQ  Control Systems} 

\author{Marina Delgado--T\'ellez}\address{Depto. de Matem\'atica Aplicada Arquitectura T\'ecnica,  Univ. Polit\'ecnica de Madrid.
Avda. Juan de Herrera 6, 28040 Madrid, Spain}
\email{marina.delgado@upm.es}

\author{Alberto Ibort}\address{Alberto Ibort: Departamento de Matem\'aticas\\Universidad Carlos III de Madrid\\
Avda. de la Universidad 30, Legan\'{e}s 28911\\Madrid\\Spain}\email{albertoi@math.uc3m.es}

\thanks{Research partially supported by the Spanish MICINN grant MTM 2010-21186-C02-02.}

\subjclass[2000]{49J15, 34A09, 34K35, 65F10} \keywords{LQ optimal control, singular optimal control, implicit
differential equations, Hamiltonian systems, presymplectic systems, Pontryaguine's Maximum Principle, constraints algorithm, numerical algorithms.}

\begin{abstract}
A Hamiltonian algorithm, both theoretical and numerical, to obtain the reduced equations implementing Pontryagine's Maximum Principle for singular linear--quadratic optimal control problems is presented.   This algorithm is inspired on the well--known Rabier--Rheinhboldt constraints algorithm used to solve differential--algebraic equations.  Its geometrical content is exploited fully by implementing a Hamiltonian extension of it which is closer to Gotay--Nester presymplectic constraint algorithm used to solve singular Hamiltonian systems.   
Thus, given an optimal control problem whose optimal feedback is given in implicit form, a consistent set of equations is obtained describing the first order differential conditions of Pontryaguine's Maximum Principle.  Such equations are shown to be Hamiltonian and the set of first class constraints corresponding to controls that are not determined, are obtained explicitly.  

The strength of the algorithm is shown by exhibiting a numerical implementation with partial feedback on the controls that provides a partial optimal feedback law for the problem.
The numerical algorithm is inspired on a previous
analysis by the authors of the consistency conditions for singular linear--quadratic optimal
control problems.    The numerical algorithm is explicitly Hamiltonian.  It is shown that only two possibilities can arise for the reduced Hamiltonian PMP: that the
reduced equations are completely determined or that they depend on a family of free parameters  that will be called the ``gauge'' controls of the system determined by the first class constraints. 
The existence of these extra degrees of freedom opens new possibilities for the search of solutions.
Numerical evidence of the stability of the algorithm is presented by discussing various relevant numerical experiments.
\end{abstract}

%%%%%%%%%%%%%%%%%%%%%%%%

\maketitle

%%%%%%%%%%%%%%%%%%%%%%%%

\tableofcontents

%\newpage

%%%%%%%%%%%%%%%%%%%%%%%%%

\section{Introduction}

In a recent paper \cite{De09} it was shown that the reduced first order differential conditions of Pontryagine's 
Maximum Principle (PMP) for singular optimal control LQ problems can be obtained by using a stable
linear numerical algorithm inspired in the theory of constraints whose origin lies in the theory of 
singular Lagrangians on analytical mechanics.    However that numerical algorithm was not 
compatible in general with the natural Hamiltonian structure of those systems.     

In this paper we will present an algorithm to obtain the differential conditions implementing PMP for singular optimal control problems (both theoretically and numerically) that is inspired too in the theory of constraints but is explicitly Hamiltonian.     In order to achieve it, a new Hamiltonian structure is introduced 
that provide an explicit Hamiltonian structure for the PMP.  Such Hamiltonian structure is the
linear counterpart of the construction known in symplectic geometry as the coisotropic embbeding
of presymplectic manifolds \cite{Go81}.  This extended Hamiltonian structure allows in addition to analyze the reduced PMP by using the standard classification of the constraints of the problem in first and second class.

We will recall that an optimal control problem is singular if there is no optimal feedback law or, in
other words, if the first order  differential conditions imposed by Pontryagine's Maximum Principle
define an implicit system of  differential equations.   

It is natural to approach such situation by
considering a reduction of the system imposing the consistency of the corresponding set of differential-algebraic equations.  Such procedure
can be carried out in various forms (see for instance \cite{Ku06} and references therein).  A geometrical way of describing the consistency of implicit systems was discussed by Rabier and Rheinboldt
\cite{Ra94} (see also \cite{Me95}).    

These algorithms are broad generalizations of Gotay's and Nester presymplectic constraint algorithm (PCA)
\cite{Go78} that allows to find the solutions of presymplectic Hamiltonian systems and that provides a
geometrical framework to the Dirac--Bergmann  theory of constraints for singular Lagrangian
systems, i.e., Lagrangian systems such that the Legendre transform is not a diffeomorphism \cite{Di50}.
L\'opez and Mart\'{\i}nez \cite{Lo00} used it to discuss problems in subriemannian geometry.
Volckaert already started the use of the PCA algorithm to study singular optimal control problems \cite{Vo99} and Petit
\cite{Pe98} used the Rabier-Rheinboldt algorithm to solve a class optimal control problems with control
equation in implicit form.   Delgado and Ibort proposed an analysis of singularities and the use of
constraint algorithms in the realm of optimal control theory \cite{De03a, De03b} and more recently
Barbero {\em et al} \cite{Ba09} have discussed the PCA algorithm in this setting again.   

The numerical implementation of the constraint algorithm discussed in \cite{De09} produce the largest 
subspace where the implicit differential equations defined by the PMP are consistent and in the setting of LQ optimal control provides an alternative analysis to the index and the structure of the system \cite{Ba06}, \cite{Ku97}.    The constraint
algorithm becomes a linear algebra problem and can be suitably dealt with by an appropriate use of
the SVD algorithm as in \cite{De09}, however such reduction algorithm didn't take advantage of all 
the information provided by the SVD decomposition.  In fact at any step of the algorithm
a number of control variables can be solved explicitly and reinserted in the algorithm.
We will call such mechanism a constraints algorithm with partial feedback on the controls and will be developed in Sections 1 and 2 of this paper.    

Moreover, as it was indicated above, the general numerical solution of DAE's (see for instance \cite{Br89} and references therein) doesn't preserve the Hamiltonian (or symplectic) structure of the problem.    In
Section 3 the presymplectic geometry of the problem is reviewed and a new Hamiltonian regularization of the problem is presented that will be used to develop a Hamiltonian constraints algorithm that extends the constraints algorithm with partial feedback, and thus it is shown to be Hamiltonian.  The Hamiltonian extension of the constraints algorithm allows for a complete description of the structure of the reduced PMP by introducing the notion of first and second class constraints.   It will be shown that the reduced equations are not completely determined unless there are no first class constraints and that the undetermined degrees of freedom are just determined explicitly by them.

Section 4 analyzes the numerical implementation of this Hamiltonian regularization
for singular LQ optimal control problems.
In Section 5 the pseudocodes of the various algorithms are discussed in detail and,
finally, in Section 6 various relevant numerical experiments are presented that explore some of the different
situations that arise in the theory and that provide evidence for the stability of
the algorithm.

%%%%%%%%%%%%%%%%%%%%%%%%%%%%%%%%%%%%%%%%%%%%%%%%%
%%%%%%%%%%%%%%%%%%%%%%%%%%%%%%%%%%%%%%%%%%%%%%%%%

\section{A recursive constraints algorithm with partial feedback on the controls for singular LQ systems}\label{sing-lq}

%%%%%%%%%%%%%%%%%%%%%%%%%%%%%%%%%%%%
%%%%%%%%%%%%%%%%%%%%%%%%%%%%%%%%%%%%

\subsection{The setting}
As it was stated in the introduction, we will study LQ optimal control
systems (even if many of the theoretical results obtained here remain valid in more general settings), more specifically we will discuss the problem of finding $C^1$-piecewise smooth curves
$\gamma(t)=(x(t),u(t))$ satisfying the linear control equations\footnote{Einstein's summation convention on repeated indices will be used in what follows unless stated explicitly.}:
\begin{equation}\label{control}
    \dot{x}^i=A^i_j\,x^j+B^i_a\,u^a,  \quad i = 1, \ldots, n,\quad  a = 1, \ldots, r,
\end{equation}
minimizing the objective functional:
\begin{equation}
    \label{cost} S(\gamma)=\int_{t_0}^TL(x(t),u(t))\,d t,
\end{equation}
where the quadratic Lagrangian $L$ has the form:
\begin{equation}
    L(x,u)=\frac{1}{2}Q_{ij}\,x^i\,x^j+N_{ia}\,x^i\,u^a+\frac{1}{2}R_{ab}\,u^a\,u^b,
\end{equation}
the matrices $Q$ and $R$ being symmetric, and subjected to fixed endpoint conditions: $x(t_0)=x_0$,
$x(T)=x_T$. The coordinates $x^i$, $i=1,\ldots,n$ describe points in state space  $x\in \R^n$, and
$u^a$, $a=1,\ldots,m$, are the control parameters  defined on the linear control space $\R^m$. The
matrices $A$, $B$, $Q$, $N$ and $R$ will be considered to be constant for simplicity even if the
algorithm  we are going to discuss can be applied without difficulty to the time-dependent
situation.

Normal extremals are provided by Pontryagine's Maximum
Principle \cite{Po62}: the curve $(x(t),\,u(t))$ is a normal extremal if there exists a lifting
$(x(t), p(t))$ of $x(t)$ to the costate space $\R^n\times\R^n$ satisfying Hamilton's
equations:
\begin{equation}\label{pmp}
    \dot{x}^i=\frac{\partial H}{\partial p_i}, \quad \quad
    \dot{p}_i=-\frac{\partial H}{\partial x^i},
\end{equation}
where $H$ is  Pontryagine's Hamiltonian function:
\begin{equation}\label{ham}
    H(x,p,u) = p_i\,(A^i_j\,x^j+B^i_a\,u^a)-L(x,u),
\end{equation}
and the set of conditions:
\begin{equation}\label{liga}
    \phi^{(1)}_a (x,p,u):= \frac{\partial H}{\partial u^a}
    =p_i\,B^i_a-N_{ia}\,x^i-R_{ab}\,u^b=0.
\end{equation}

%%%%%%%%%%%%%%%%%%%%%%%%%%%%%%%%%%%%
%%%%%%%%%%%%%%%%%%%%%%%%%%%%%%%%%%%%

\subsection{Consistency conditions and the recursive constraints algorithm}\label{consistency}

We will call conditions in Eq. (\ref{liga}) primary constraints.  Trajectories   that are
solution to the optimal control problem must lie in the linear submanifold:
\begin{equation}\label{m1}
    M_1=\{(x,p,u)\in M_0|~\phi^{(1)}_a (x,p,u) = 0 \},
\end{equation}where $M_0 = \set{(x,p,u)\in\R^{2n+m}}$ denotes the total  space of the system.

If $\gamma(t) = (x(t),p(t),u(t))$ is a solution of the optimal problem, then its derivative will satisfy Eqs. \eqref{pmp} and
\begin{equation}\label{deru}
\dot{u}^a = C^a(x,p,u),
\end{equation}
with $C^a(x(t),p(t),u(t))$ a function depending on the curve $\gamma(t)$, however because of the constraint \eqref{liga}
we get,
$$\pd{\phi^{(1)}_a}{x^i}\pd{H}{p_i}-\pd{\phi^{(1)}_a} {p_i}\pd{H}{x^i}+
     \pd{\phi^{(1)}_a}{u^b}\,C^b=0.$$
Then, if the matrix
$$R_{ab}=\frac{\partial^2 H} {\partial u^a\partial u^b}$$
is invertible in $M_1$,  then the functions $C^a$ are completely determined and
there exists an optimal feedback condition solving Eq. (\ref{liga}), given explicitly by:
\eq{\label{feedback} u^b=(R^{-1})^{ab}\,(p_iB^i_{\,a}-N_{ia}x^i).}%
and then, we obtain for Eq. (\ref{deru}):
\eqi{\dot{u}^b&=&(R^{-1})^{ab}(\dot{p}_iB^i_a-N_{ia}\dot{x}^i)=\\
%(R^{-1})^{ab}\left(-\pd{H}{x^i}B^i_a-N_{ia}\pd{H}{p_i}\right)=\\
&=&(R^{-1})^{ab}\left((-p_jA^j_i+Q_{ij}x^j+N_{id}u^d)B^i_a-N_{ia}(A^i_j\,x^j+B^i_d\,u^d)\right).}
Notice that in this case $\dot{\phi}^{(1)}_a$ vanishes automatically on $M_1$.
Linear-quadratic systems such that the matrix $R$ is invertible will be called in what follows regular and singular otherwise.

For a singular optimal LQ  system it may occur that solutions of (\ref{pmp}) starting at points
$(x_0, p_0,u_0)\in M_1$ will not be contained in
$M_1$ for $t> 0$ for any $u$.   Because Eqs. (\ref{pmp})-(\ref{liga}) must be satisfied along
optimal paths, we must consider initial conditions only those points  $(x_0,p_0,u_0)$ for which
there is at least a solution of (\ref{pmp}) contained in $M_1$ starting from them, i.e.,  such that there
exists $C$ for which,
\begin{equation}\label{recurr}
 \phi^{(2)}_a:=\dot{\phi}^{(1)} _a =0,
\end{equation}%
where the derivative is taken in the direction of Eq. \eqref{pmp} and $\dot{u}=C$. The subset
obtained, that will be denoted by $M_2$, is again a linear submanifold of $M_1$ (this is also true
in the time-dependent case).   The functions $\phi^{(2)} _{\ \ a}$ defining  $M_2$ in $M_1$ will be called 
secondary constraints, also known as ``hidden
constraints'' in the context of DAEs. 

Clearly, the argument goes on and we will obtain in this form a family of linear
submanifolds defined recursively as follows:%
\begin{equation}\label{recursive} 
M_{k+1}=\{(x,p,u)\in M_k \mid \, \exists\ C~{\mbox{such that}}~Ê
\phi^{(k+1)} _a := \dot{\phi}^{(k)} _a\,(x,p,u) =0\} , \quad k \geq 1.
\end{equation}
where the derivative $\dot{\phi}^{(k)} _a \,(x,p,u)$ it taken in the direction of eq. \eqref{pmp} and $\dot{u}^a=C^a$ on $M_k$, i.e., satisfying all previous constraints $\phi^{(l)}$, $l < k$. 
Eventually, the recursion will stop and $M_\kappa = M_{\kappa+1} = M_{\kappa+2} = \dots$, for certain finite $\kappa$.  We will call the number $\kappa$ of steps it takes the algorithm to stabilize, the (recursive) index of the problem.  The final space $M_\kappa := M_\kappa$ will be called the final constraint submanifold of the problem, that is the set of consistent initial condition for the DAE (\ref{pmp})-(\ref{liga}).

As it was stressed in the introduction, this algorithm constitutes an adapted version to the present setting of the reduction algorithm for DAEs \cite{Ra94,Ra02} (see also \cite{Ku06} and \cite{Br89}).
However, the DAE system above has an additional geometrical structure.   We will devote the following sections to incorporate this additional structure in the algorithm, but before doing it we will discuss a version of this algorithm that incorporates a partial feedback of the controls.  The idea behind it is that depending on the rank of the matrix of coefficients of the controls in the constraints, part of them could be solved providing a partial optimal feedback law.  This idea will be discussed in the following section.

%%%%%%%%%%%%%%%%%%%%%%%%%%%%%%%%%%%%
%%%%%%%%%%%%%%%%%%%%%%%%%%%%%%%%%%%%

\subsection{Constraints algorithm with partial feedback on the controls}\label{pcasect}

We will use matrix notation in order to write in a compact form the recursive law for the algorithm.

%%%%%%%%%%%%%%%%%%%%%%%%%%%%%%%%%%%%

\subsubsection{The initial data}
We will denote by $M^{(0)} := M_0 = \mathbb{R}^{2n} \times \mathbb{R}^m$ the total space of the system where $x$, $p \in \mathbb{R}^n$ and $u^{(0)}:= u\in \mathbb{R}^{m_0}$ ($m_0 := m$) denote column vectors.

Thus the problem we have is to
obtain solutions of the differential equation:
\eqi{\dot{x} & = & Ax+Bu,}with $A,~Q\in\R^{n\times n}$, $B,~N\in\R^{n\times m}$ and
$R\in\R^{m\times m}$, such that Pontryagine's Hamiltonian is maximized:
\eqi{H(x,p,u)=p^TAx+p^TBu-\mitad x^TQx-x^TNu-\mitad u^TRu.}%
Thus, the differential conditions for the PMP, eq. \eqref{pmp} become:
$$
\dot{x} = \pd{H}{p^T}=Ax+Bu, \quad \dot{p} = -\pd{H}{x^T} = Qx-A^Tp+Nu,\quad \dot{u}  = C.
$$
We may written the previous differential equation as a vector field depending on the parameters $u$
as:
$$
\Gamma^{(0)}(u) = \left[ \begin{array}{c} \dot{x} \\ \dot{p} \end{array} \right] = \left[ \begin{array}{c} Ax+Bu \\ Qx-A^Tp+Nu \end{array} \right] = G^{(0)} \left[ \begin{array}{c} x \\ p \end{array} \right] + Z^{(0)} u^{(0)} ,
$$
with 
$$ 
 G^{(0)} = \left[ \begin{array}{cc} A & 0  \\ Q & - A^T \end{array} \right], \quad Z^{(0)} = \left[ \begin{array}{c} B \\ N \end{array} \right] 
 $$
The primary constraint, eq. \eqref{liga}, written as column vector is given by:
$$
\phi^{(1)} :=  -N^Tx+B^Tp-Ru = S^{(1)} \left[ \begin{array}{c} x \\ p \end{array} \right] - R^{(1)} u^{(0)}
$$
with 
\begin{equation}\label{S1}
S^{(1)} = [-N^T | B^T] = -Z^{(0)\,T} J, \quad R^{(1)} = R ,
\end{equation}
and $J$ denotes the $2n\times 2n$ canonical symplectic matrix:
$$
J = \left[ \begin{array}{c|c} 0 & I_n \\ \hline -I_n & 0 \end{array} \right].
$$
Because of the linear nature of the equations above, all further constraints derived in the algorithm will be linear.   

%%%%%%%%%%%%%%%%%%%%%%%%%%%%%%%%%%%%

\subsubsection{The partial feedback}
If $R^{(1)}$ is singular we implement the recursive constraints algorithm as follows.  
First we compute the SVD of $R^{(1)}$
$$
R^{(1)} = U^{(1)}\, \left[\begin{array}{c|c}
\Sigma^{(1)} & 0 \\  \hline 0 &0
\end{array}\right]\,V^{(1)\,T} .
$$
with $\Sigma^{(1)}$ the diagonal matrix of rank $r_1$ whose diagonal elements are the singular values $s_1, \ldots s_{r_1}$ of $R^{(1)}$.
We will redefine the controls $u^{(0)}$ as follows:
$$ 
\tilde{u}^{(0)} := (V^{(1)\, T} u^{(0)}) (1:r_1), \quad u^{(1)} := (V^{(1)\, T} u^{(0)}) (r_1+1:m_0),
$$
in other words:
\begin{equation}\label{u0}
\left[\begin{array}{c}
\tilde{u}^{(0)} \\  \hline u^{(1)}
\end{array}\right] = V^{(1)\, T} u^{(0)} 
\end{equation}

%Then we can write the constraints obtained at step $k$ of the algorithm as:
%\begin{equation}\label{ligak}
%\phi^{(k)}(x,p,u) := \sigma^{(k)}x + \beta^{(k)}p + \rho^{(k)}u^{(k)} ,
%\end{equation}
%where $\sigma^{(k)}, \beta^{(k)} \in \R^{r_k\times n}$, $\rho^{(k)} \in \R^{r_k\times m_k}$, for some
%$r_k, m_k\in\N$.  The size $m_k$ of the vector $u^{(k)} \in \mathbb{R}^{m_k\times 1}$ could change %from step to step because of the feedback substitution that
%will be performed at each step (see later).
%Notice that
%$$ \sigma^{(1)} = -N^T, \quad \beta^{(1)} = B^T, \quad \rho^{(1)} = - R .$$%

%The matrix equations $\phi^{(k)} = 0$, define the linear submanifolds $M_k$ obtained by applying the %recursive constraint algorithm. Thus,
%the matrices $\sigma^{(k)}$, $\beta^{(k)}$, $\rho^{(k)}$ completely characterize the constraints
%$\phi^{(k)}$.   We will
%use the notation $[\sigma| \beta]$ with $[\sigma|\beta](k,:) := [\sigma|\beta]^{(k)} = [\sigma^{(k)}, \beta^%{(k)}]$ and we write the
%constraints compressed in matrix form: 
%\begin{equation}\label{phik}
%\phi^{(k)} = [\sigma|\beta]^{(k)} \matrix{c}{x\\\hline p} + \rho^{(k)} u .
%\end{equation}

Then multiplying the primary constraint $\phi^{(1)}$ by $U^{(1)\,T}$ on the left:
\begin{eqnarray}\label{feedliga}
0&=&U^{(1)\,T}\,\phi^{(1)} = U^{(1)\,T} S^{(1)} \matrix{c}{x\\\hline p} -
\left[\begin{array}{c|c} \Sigma^{(1)} & 0 \\  \hline 0 &0 \end{array}\right] \left[\begin{array}{c}
\tilde{u}^{(0)} \\  \hline u^{(1)} \end{array} \right] \\
&=& \nonumber U^{(1)\,T} S^{(1)}\matrix{c}{x\\\hline p}+ \left[\begin{array}{c}
\Sigma^{(1)}\tilde{u}^{(0)} \\  0 \end{array} \right] 
\end{eqnarray} %
The previous equation splits in two parts:

\begin{enumerate}
\item The partial feedback: 
\begin{equation} 
\tilde{u}^{(0)} =  (\Sigma^{(1)})^{-1}U^{(1)\,T}(1:r_1,:)S^{(1)}\matrix{c}{x\\\hline
p} =: \feed^{(1)}\matrix{c}{x\\\hline p},
\end{equation}

\item The reduced primary constraint:
\begin{equation}\label{phi2}
\widetilde{\phi}^{(1)} = U^{(1)\,T}(r_1+1:m,:)S^{(1)}\matrix{c}{x\\\hline p} = 0 
\end{equation}%
\end{enumerate}
We may write:
$$
S_f^{(1)} = U^{(1)\,T}(1:r_1,:)S^{(1)}, \quad S_c^{(1)} =U^{(1)\,T}(r_1+1:m,:)S^{(1)} ,
$$
this is,
$$
\left[\begin{array}{c} S_f^{(1)} \\ \hline S_c^{(1)} \end{array} \right] = U^{(1)\,T} S^{(1)}.
$$
With this notation, the matrix $\feed^{(1)}$ determining partial feedback of the controls $\tilde{u}^{(0)}$ reads:
$$
\feed^{(1)} =   (\Sigma^{(1)})^{-1} S_f^{(1)} ,
$$
and the reduced primary constraint $\widetilde{\phi}^{(1)}$:
\begin{equation}\label{rphi1}
\widetilde{\phi}^{(1)} = S_c^{(1)} \matrix{c}{x\\\hline p} = 0 .
\end{equation}

If we denote by $m_1$ the length of the vector $u^{(1)}$, i.e., $m _1 = m_0 - r_1$, then
we may consider that our optimal control problem is defined on the space
$M^{(1)} = \mathbb{R}^{2n} \times \mathbb{R}^{m_1}$ with elements $(x,p,u^{(1)})$.
There the reduced primary constraint $\widetilde{\phi}^{(1)}$ defines a linear submanifold $\widetilde{M}_1 = \{ (x,p,u^{(1)}) \in M^{(1)} \mid \widetilde{\phi}^{(1)} = 0 \}$.    This submanifold $\widetilde{M}_1$ is clearly isomorphic to $M_1$ by using the map:
\begin{equation}\label{alpha1}
\alpha_1 \colon \widetilde{M}_1 \to M_1,\quad  \alpha_1(x,p,u^{(1)}) = (x,p,u^{(0)})
\end{equation}
with:
$$
u^{(0)} = \left[\begin{array}{c} \tilde{u}^{(0)} \\ \hline u^{(1)} \end{array} \right] , \quad \tilde{u}^{(0)} = \feed^{(1)}\left[\begin{array}{c} x \\ \hline p \end{array} \right]
$$

%%%%%%%%%%%%%%%%%%%%%%%%%%%%%%%%%%%%

\subsubsection{The first iteration}\label{first_iteration}
The identification between $\widetilde{M}_1$ and $M_1$ allows to write the differential equation $\Gamma^{(0)}(u)$ as a vector field on $\widetilde{M}_1$, in other words, if we introduce the feedback for the controls $\tilde{u}^{0}$ on the equations for $\dot{x}$ and $\dot{p}$ become equations just on the controls $u^{(1)}$. We will denote the vector field thus defined on $\widetilde{M}_1$ as $\Gamma^{(1)}(u^{(1)})$ (or simply $\Gamma^{(1)}$ if it is not necessary to make explicit the dependence on the control parameters $u^{(1)}$), then:
$$
\Gamma^{(1)} = G^{(0)} \left[ \begin{array}{c} x \\ p \end{array} \right] + Z^{(0)} V^{(1)} \left[\begin{array}{c} \tilde{u}^{(0)} \\ \hline u^{(1)} \end{array} \right] .
$$
Writting 
$$
\tilde{Z}^{(0)} := (Z^{(0)} V^{(1)})(1:r_1,:) ,\quad   Z^{(1)} := (Z^{(0)} V^{(1)})(r_1+1:m_0,:) ,
$$
this is:
$$
\left[\begin{array}{c} \tilde{Z}^{(0)} \\ \hline Z^{(1)} \end{array} \right]  = Z^{(0)} V^{(1)}
$$
we will get finally:
$$
\Gamma^{(1)} = G^{(0)} \left[ \begin{array}{c} x \\ p \end{array} \right] + \tilde{Z}^{(0)} \tilde{u}^{(0)} +  Z^{(1)} u^{(1)} = G^{(1)}  \left[ \begin{array}{c} x \\ p \end{array} \right]  +Z^{(1)} u^{(1)} ,
$$
with 
$$
G^{(1)} = G^{(0)} + \tilde{Z}^{(0)} \feed^{(1)} .
$$

%\eqi{\matrix{c}{\dot{x}\\\hline \dot{p}}=\matrix{c|c}{A_x&A_p\\\hline Q_x&Q_p}\matrix{c}{x\\\hline p}+
%\matrix{c}{B_{12}\\\hline N_{12}} u^{(2)} =: VF^{(1)}\matrix{c}{x\\\hline p}+ BN_u^{(1)} u^{(2)},}
%with%
%\eqi{VF^{(1)}&:=&\matrix{c|c}{A_x&A_p\\\hline Q_x&Q_p}=\matrix{c|c}{A&0\\ \hline Q&-A^T}+
%\matrix{c}{B_{11}\\ \hline N_{11}}\feed_{xp}^{(1)},}
%that is,
%\eqi{VF^{(1)}&=&VF^{(0)}+BN_1^{(1)}\feed_{xp}^{(1)}.}

Recall that the consistency of the differential equations given by $\Gamma^{(0)}(u)$ on $M_1$ or equivalently, by $\Gamma^{(1)}$ on $\widetilde{M}_1$, determines the secondary constraints, i.e., the submanifold $M_2$, etc.   
Thus imposing consistency to the reduced primary constraint $\tilde{\phi}^{(1)}$, eqs. \eqref{phi2} or \eqref{rphi1}, we get the secondary constraint:
\begin{equation}\label{secondary_feed}
0 = \phi^{(2)} = \dot{\widetilde{\phi}}^{(1)} = S_c^{(1)}\matrix{c}{\dot{x}\\\hline \dot{p}} 
= S_c^{(1)}\bigg( G^{(1)}\matrix{c}{x\\\hline p} + Z^{(1)} u^{(1)}\bigg) = S^{(2)}\matrix{c}{x\\\hline p}  - R^{(2)} u^{(1)}.
\end{equation}
with the new coefficients:
\begin{equation} 
 S^{(2)} = S_c^{(1)}G^{(1)},\quad R^{(2)} = -S_c^{(1)} Z^{(1)} .
 \end{equation}
At this point, we will iterate the procedure, this is, if $R^{(2)}$ is regular, we will solve the controls $u^{(1)}$ and the algorithm stops.   If $R^{(2)}$ is singular we will compute its SVD again:
$$
R^{(2)} = U^{(2)}\, \left[\begin{array}{c|c} \Sigma^{(2)} & 0 \\  \hline 0 &0 \end{array}\right]\,V^{(2)\,T} .
$$
with $\Sigma^{(2)}$ or rank $r_2$, and we decompose the controls $u^{(1)}$ as:
$$
\left[\begin{array}{c}
\tilde{u}^{(1)} \\  \hline u^{(2)}
\end{array}\right] := V^{(2)\, T} u^{(1)} 
$$
where $\tilde{u}^{(1)}$ has length $r_2$ and $u^{(2)}$ has lenght $m_2 = m_1 - r_2 = m_0 - r_1 - r_2$.
Then we will obtain the decompositions:
$$
U^{(2)\, T} S^{(2)} = \matrix{c}{ S_f^{(2)}\\\hline S_c^{(2)}} , \quad 
Z^{(1)} V^{(2)} = [\tilde{Z}^{(1)} \mid Z^{(2)} ] .
$$
Then we will decompose the modified secondary constraint $U^{(2)\, T} \phi^{(2)}$ into the feedback part and the new reduced constraint $\widetilde{\phi}^{(2)}$, obtaining respectively:
$$ 
\tilde{u}^{(1)} = \feed^{(2)} \matrix{c}{ x \\\hline p}, \quad \feed^{(2)}=  (\Sigma^{(2)})^{-1} S_f^{(2)} ,
$$
and
$$
\widetilde{\phi}^{(2)} = S_c^{(2)}  \matrix{c}{ x \\\hline p} .
$$
Then, the vector field $\Gamma^{(2)}(u^{(2)})$ is given by:
$$
\Gamma^{(2)} = G^{(2)} \matrix{c}{ x \\\hline p} + Z^{(2)} u^{(2)}, \quad G^{(2)} = G^{(1)} + 
\tilde{Z}^{(1)} \feed^{(2)} ,
$$
and the stability condition for the reduced secondary constrains $\widetilde{\phi}^{(2)}$ gives the tertiary constraints:
$$
0 = \phi^{(3)} = \dot{\widetilde{\phi}}^{(2)} = \Gamma^{(2)}(\widetilde{\phi}^{(2)}) = S^{(3)} \matrix{c}{ x \\\hline p} - R^{(3)}u^{(2)}, 
$$
where
$$
S^{(3)}= S_c^{(2)} G^{(2)}, \quad R^{(3)} = - S_c^{(2)} Z^{(2)} ,
$$
and the next iteration follows easily.

%%%%%%%%%%%%%%%%%%%%%%%%%%%%%%%%%%%%

\subsubsection{The general iteration}
At the $k$th iteration of the algorithm, the matrices:
$$
G^{(k)}, Z^{(k)}, U^{(k)}, V^{(k)}, S^{(k)}
$$
will be given, as well as the decompositions:
$$
U^{(k)\, T} S^{(k)} = \matrix{c}{ S_f^{(k)}\\\hline S_c^{(k)}} ,
$$ 
and
\begin{equation}\label{Zk+1} 
Z^{(k)} V^{(k)} = [\tilde{Z}^{(k)} \mid Z^{(k+1)} ] .
\end{equation}
The vector field $\Gamma^{(k)}$ will be given by:
\begin{equation}\label{gamma_k}
\Gamma^{(k)} = G^{(k)} \matrix{c}{ x \\\hline p} + Z^{(k)} u^{(k)}, 
\end{equation}
and the reduced $k$--ary constraints will be:
$$
\widetilde{\phi}^{(k)} = S_c^{(k)}  \matrix{c}{ x \\\hline p} .
$$
The reduced constraints $\widetilde{\phi}^{(k)}$ define a linear submanifold $\tilde{j}_k \colon \widetilde{M}_k \to M^{(k)} = \mathbb{R}^{2n}\times \mathbb{R}^{m_k}$ :
\begin{equation}\label{Mtilde_k}
\widetilde{M}_k = \{ (x,p,u^{(k-1)}) \in M^{(k)}= \mathbb{R}^{2n}\times \mathbb{R}^{m_k} \mid \widetilde{\phi}^{(k)} = 0  \} ,
\end{equation}
which is isomorphic to the linear submanifold $M_k$ by means of the map:
\begin{equation}\label{alphak}
\alpha_k \colon \widetilde{M}_k \to M_k,\quad  \alpha_k(x,p,u^{(k)}) = (x,p,u^{(k-1)})
\end{equation}
with:
$$
u^{(k-1)} = \left[\begin{array}{c} \tilde{u}^{(k-1)} \\ \hline u^{(k)} \end{array} \right] , \quad \tilde{u}^{(k-1)} = \feed^{(k)}\left[\begin{array}{c} x \\ \hline p \end{array} \right] .
$$

The consistency condition for the $k$--ary reduced constraints will give the $k+1$--ary constraints:
\begin{equation}\label{consistency_reducedk}
0 = \phi^{(k+1)} = \dot{\widetilde{\phi}}^{(k)} = \Gamma^{(k)}(\widetilde{\phi}^{(k)}) = S^{(k+1)} \matrix{c}{ x \\\hline p} - R^{(k+1)}u^{(k)}, 
\end{equation}
with
\begin{equation}\label{Sk+1}
S^{(k+1)}= S_c^{(k)} G^{(k)}, \quad R^{(k+1)} = - S_c^{(k)} Z^{(k)} .
\end{equation}
After computing the SVD factorization of $R^{(k+1)}$:
\begin{equation}\label{svd+1}
R^{(k+1)} = U^{(k+1)}\, \left[\begin{array}{c|c} \Sigma^{(k+1)} & 0 \\  \hline 0 &0 \end{array}\right]\,V^{(k+1)\,T} .
\end{equation}
with $\Sigma^{(k+1)}$ or rank $r_{k+1}$, we decompose the controls $u^{(k)}$ as:
$$
\left[\begin{array}{c}
\tilde{u}^{(k)} \\  \hline u^{(k+1)}
\end{array}\right] := V^{(k+1)\, T} u^{(k)} 
$$
where $\tilde{u}^{(k)}$ has length $r_{k+1}$ and $u^{(k+1)}$ has lenght $m_{k+1} = m_k - r_{k+1} = m_0 - (r_1 + \cdots + r_{k+1})$.
There again we will decompose the matrix $S^{(k+1)}$ according with:
$$
U^{(k+1)\, T} S^{(k+1)} = \matrix{c}{ S_f^{(k+1)}\\\hline S_c^{(k+1)}}.
$$
Thus the partial feedback for the controls $\widetilde{u}^{k}$ will be:
\begin{equation}\label{feed_controlsk}
\tilde{u}^{(k)} = \feed^{(k+1)} \matrix{c}{ x \\\hline p}, \quad \feed^{(k+1)}=  (\Sigma^{(k+1)})^{-1} S_f^{(k+1)} ,
\end{equation}
the reduced $(k+1)$--ary constraints will be given by:
\begin{equation}\label{k+1reduced}
\widetilde{\phi}^{(k+1)} = S_c^{(k+1)}  \matrix{c}{ x \\\hline p} ,
\end{equation}
and the new vector field:
$$
\Gamma^{(k+1)} = G^{(k+1)} \matrix{c}{ x \\\hline p} + Z^{(k+1)} u^{(k+1)}, 
$$
where
\begin{equation}\label{Gk+1}
G^{(k+1)} = G^{(k)} + \tilde{Z}^{(k)} \feed^{(k+1)} .
\end{equation}
Now we are ready to proceed again as we have obtained the matrices (eqs. \eqref{Gk+1}, \eqref{Zk+1}, \eqref{svd+1}, \eqref{Sk+1}):
$$ 
G^{(k+1)}, Z^{(k+1)}, U^{(k+1)}, V^{(k+1)}, \Sigma^{k+1}, S^{(k+1)} ,
$$
that together with $G^{(1)}, Z^{(1)}, U^{(1)}, V^{(1)}, \Sigma^{1}, S^{(1)}$ obtained in Section \ref{first_iteration} will complete the recursive constraints algorithm with partial feedback on the controls.

\subsubsection{The total feedback}
At the end, if $\kappa$ is the index of the problem, the total feedback is the collection of all partial feedbacks $\widetilde{u}^{(0)},\ldots, \widetilde{u}^{(\kappa)}$ ($u= u^{(0)}, m = m_0$):
\begin{eqnarray*}
\tilde{u}^{(0)}(1:r_1)&=&V^{(1)\,T}(1:r_1,:)u = \feed^{(1)}\matrix{c}{x\\\hline p}\\
u^{(1)}&:=& V^{(1)\,T}(r_1+1:m,:)u\\
\tilde{u}^{(1)}(1:r_2)&=&V^{(2)\,T}(1:r_2,:) u^{(1)}=\\&=&V^{(2)\,T}(1:r_2,:)V^{(1)\,T}(r_1+1:m,:)u=\feed^{(1)}\matrix{c}{x\\\hline p}\\
u^{(2)}&:=& V^{(2)\,T}(r_2+1:m_1,:)u^{(1)}=\\
&=&V^{(2)\,T}(r_2+1:m_1,:)V^{(1)\,T}(r_1+1:m,:)u\\
&\dots&\\
\tilde{u}^{(\kappa)}(1:r_K)&=&V^{(\kappa+1)\,T}(1:r_{\kappa+1},:) u^{(\kappa)}=\\
&=&V^{(\kappa+1)\,T}(1:r_{\kappa+1},:)V^{(\kappa)\,T}(r_\kappa+1:m_\kappa,:)\cdots\\
&\cdots& V^{(2)\,T}(r_2+1:m_2,:)V^{(1)\,T}(r_1+1:m_1,:)u=\feed^{(\kappa)}\matrix{c}{x\\\hline p}\\
u^{(\kappa+1)}&:=& V^{(\kappa+1)\,T}(r_{\kappa+1}+1:m_K,:) u^{(\kappa)}=\\
&=&V^{(\kappa+1)\,T}(r_{\kappa+1}+1:m_\kappa,:)\cdots V^{(2)\,T}(r_2+1:m_1,:)V^{(1)\,T}(r_1+1:m,:)u
\end{eqnarray*}%
The controls that remain free at the end, i.e., the vector $u^{(K+1)}$, which are in a number $m - (r_1 + \cdots + r_{\kappa+1})$, can be written as $u^{(\kappa+1)}=\nofeed\cdot u$, where the matrix $\nofeed$ is:
$$
\nofeed:=V^{(\kappa+1)\,T}(r_{\kappa+1}+1:m_\kappa,:)\cdots V^{(2)\,T}(r_2+1:m_1,:)V^{(1)\,T}(r_1+1:m,:) .
$$ 
and the part of the controls that have been solved by using the successive partial feedbacks can be written as
$$
\widetilde{u} = V \cdot u=\feedtot\matrix{c}{x\\\hline p},
$$ 
with
$$
V = \left[ \begin{array}{c} 
V^{(1)\,T}(1:r_1,:) \\
V^{(2)\,T}(1:r_2,:)V^{(1)\,T}(r_1+1:m,:)\\
\vdots\\
V^{(\kappa+1)\,T}(1:r_{\kappa+1},:)\cdots V^{(1)\,T}(r_1+1:m,:) 
\end{array} \right]
$$
and
$$
\feedtot = \matrix{c}{\feed^{(1)}\\
\vdots\\
\feed^{(\kappa+1)}}.
$$

%The iteration for the algorithm at each step is the following:
%\eqi{\feedtot^{(k)}=\matrix{c}{\feedtot^{(k-1)}\\V^{(k)\,T}(1:r_k,:)\cdot \nofeed^{(k-1)}},\\\\
% \nofeed^{(k)}=V^{(k)\,T}(r_k+1:m_k,:)\nofeed^{(k-1)}.}
%The final vector field is given by:
%\eqi{\matrix{c}{\dot{x}\\\hline \dot{p}}=VF^{(K)}\matrix{c}{x\\\hline p}+ BN_u^{(K)} u^{(K+1)}.}

%%%%%%%%%%%%%%%%%%%%%%%%%%%%%%%%%%%%%%%%%%%%%%%
%%%%%%%%%%%%%%%%%%%%%%%%%%%%%%%%%%%%%%%%%%%%%%%

\section{The presymplectic structure of the PMP and the constraints algorithm}

%%%%%%%%%%%%%%%%%%%%%%%%%%%%%%%%%%%%
%%%%%%%%%%%%%%%%%%%%%%%%%%%%%%%%%%%%

\subsection{Presymplectic constraints algorithm for singular optimal control problems}\label{recursive_presymplectic}
As it was mentioned previously, the differential conditions of the PMP have the geometrical structure of a presymplectic system.   We will discuss it briefly in the particular instance of LQ systems (see for instance \cite{Ba09} for more details).     

If we denote by
$P$ the state space of our system (in the case of LQ systems $P = \mathbb{R}^n$) with local coordinates $x^i$ as above and by $C$ the space of controls  (in the case of LQ systems $C = \mathbb{R}^m$) with coordinates $u^a$, we can form the space $M_0 = T^*P \times C$ where $T^*P$ denotes the cotangent bundle of $P$ (for LQ systems $T^*P = \mathbb{R}^{2n}$) with local coordinates $(x^i, p_i)$ where $p_i$ are called costate variables.   In the space $M_0$ there is a canonical 2--form
$\Omega_0 = dx^i \wedge dp_i$ which is obtained by pulling--back along the projection map
$\pi \colon M_0 \to T^* P$, $\pi (x,p,u) = (x,p)$, the canonical symplectic form $\omega = dx^i \wedge dp_i$ defined on $T^*P$.    Notice that the 2--form $\Omega$ is obviously closed but no non--degenerate because it has a kernel $K := \ker \Omega_0$ which is spanned by the vectors $\partial /\partial u^a$.    The triple $(M_0,\Omega_0 , H_0)$ defines a presymplectic Hamiltonian system with $H_0(x,p,u)$ Pontryagine's Hamiltonian.   The dynamics is provided by any vector field $\Gamma$ such that
\begin{equation}\label{presymplectic}
i_\Gamma \Omega_0 = dH_0 .
\end{equation}
Notice that in the case of LQ systems the space $M_0$ is $\mathbb{R}^{2n+m}$, the
presymplectic 2--form $\Omega_0$ is constant.  Again in the particular instance of LQ systems, if we consider the vector field 
$$
\Gamma = \dot{x}^i \frac{\partial}{\partial x^i} + \dot{p}_i \frac{\partial}{\partial p_i} + \dot{u}^a \frac{\partial}{\partial u^a} ,  
$$ 
as a (column) vector in $\mathbb{R}^{2n+m}$ and 
$$
dH_0 = \frac{\partial H_0}{\partial x^i} dx^i + \frac{\partial H_0}{\partial p_i} dp_i + \frac{\partial H_0}{\partial u^a} du^a
$$ 
as a vector on $\mathbb{R}^{2n+m}$ using the identification between vectors and covectors provided by the given choice of a basis, the dynamical equation, eq. (\ref{presymplectic}) take the matrix form:
\begin{equation}\label{presymplectic_local}
 \left[  \begin{array}{c|c|c} 0 & -I_n & 0 \\ \hline I_n & 0 & 0  \\  \hline 0 & 0 & 0  \end{array} \right]  \left[  \begin{array}{c} \dot{x} \\ \hline \dot{p} \\ \hline \dot{u} \end{array} \right]  =  \left[  \begin{array}{c} \partial H_0 /\partial x^T \\ \hline \partial H_0 / \partial p^T \\ \hline \partial H_0 / \partial u^T \end{array} \right] ,
 \end{equation}
which is equivalent to eqs. \eqref{pmp}-\eqref{liga} (in fact the previous equation is the
local form in canonical  coordinates of any presymplectic system). 

These equations, eqs. \eqref{presymplectic} or \eqref{presymplectic_local}, have the form
of a quasilinear system on $M_0$
\begin{equation}\label{quasilinear}
A(\zeta ) \dot{\zeta} = F(\zeta )
\end{equation}
where $\zeta$ denotes the set of coordinates $(x^i, p_i, u^a)$, $A(\zeta)$ is the (constant) matrix
$[\Omega]$ and $F(\zeta )$ is the vector $\nabla H_0$.  In general a vector $\dot{\zeta}$
satisfying eq. \eqref{quasilinear} will not exists for all $\zeta$'s.   A necessary and sufficient condition for the existence of $\dot{\zeta}$ at a given point $\zeta$ is that $\langle Z, F(\zeta) \rangle = 0$, for all $Z \in\ker A(\zeta)$.  In the case of the presymplectic system eq. \eqref{presymplectic}, this condition means that at a given point $\zeta \in M_0$ the vector $\Gamma$ will exists iff $dH_0 (Z) = 0$ for all $Z \in K$.   The points at which such vector $\Gamma$ exists determines a subset $M_1$ of $M_0$.   A family of independent functions defining $M_1$ are called primary constraints $\phi^{(1)}_a$ for the presymplectic system.  
Notice that because $K$ is spanned by the vectors $Z_a= \partial/\partial u^a$, then a family of independent primary constraints are given by:
\begin{equation}\label{primary}
\phi_a^{(1)} (x,p,u) = dH_0(Z_a) = \partial H_0/\partial u^a = 0 .
\end{equation}
For LQ systems the subspace $M_1$ was discussed in Section \ref{consistency} and the primary constraints $\phi_a^{(1)}$ take the explicit expressions given in eq. \eqref{liga}.

On $M_1$ the presymplectic form $\Omega_0$ induces by restriction a new presymplectic form $\Omega_1 = \Omega_0\mid_{M_1}$.  If we denote by $i_1\colon M_1 \to M_0$ the canonical inclusion map we can also write $\Omega_1 = i_1^*\Omega_0$.   The restriction of Pontryagine's Hamiltonian to $M_1$ will be written  $H_1 = H_0 \mid_{M_1}$.   Hence we have obtained a new presymplectic system $(M_1,\Omega_1, H_1)$ on $M_1$. The corresponding dynamical equation is now given by:
$$
i_{\Gamma_1} \Omega_1 = dH_1,
$$
and we should restrict again to the subset $M_2$ of $M_1$ where $\Gamma_1$ exists.   We will repeat the argument for $M_2$ and we proceed iteratively afterwards.  

At each step $k$ of the algorithm, the presymplectic form $\Omega_0$ restricts to the submanifold $M_k$ by inducing a new presymplectic structure $\Omega_k = \Omega_0\mid_{M_k}$ on it as well as Pontryagine's Hamiltonian $H_k = H_0\mid_{M_k}$.  The dynamical equation (\ref{presymplectic}) becomes:
\begin{equation}\label{presymplectic_k}
i_\Gamma \Omega_k = dH_k,
\end{equation}
Thus the system \eqref{presymplectic_k} will be consistent, i.e., there will exist a vector $\Gamma$ satisfying eq. \eqref{presymplectic_k}, if the constrainsts defined by this system are fullfilled, this is $dH_k(Z) = 0$ for all $Z \in \ker\Omega_k$.   Let us denote by $K_k = \ker \Omega_k$.  Hence the $k$--ary constrainsts $\phi^{(k)}$ will be:
\begin{equation}\label{consistencyk}
\phi^{(k+1)} (x,p,u) = dH_k (Z) = 0, \quad \forall Z \in K_k .
\end{equation}
Thus at each step the algorithm has the structure of a presymplectic system which is consistent on the subspace:
\begin{equation}\label{recursive_sym}
M_{k+1} = \{ (x,p,u)\in M_k \mid \, dH_k (Z) = 0, \quad \forall Z \in K_k \} , \quad k \geq 1.
\end{equation}
This constitutes the recursive presymplectic constraints algorithm (or presymplectic algorithm for short) given in \cite{Go78}-\cite{Go79}.  We will denote by $i_{k+1} \colon M_{k+1} \to M_k$ the natural inclusion.  Notice that $\Omega_{k+1} = i_{k+1}^*\Omega_k$, etc.

The presymplectic algorithm above, eq. \eqref{recursive_sym} is equivalent to the recursive constraints algorithm eq. \eqref{recursive}, however it is not immediately obvious that they are equivalent to the constraints algorithm with partial feedback discussed in Section \ref{pcasect}.  The following theorem shows that they are all indeed the same.

\begin{theorem}\label{equivalence_constraints}  The presymplectic and the recursive constraints algorithms are equivalent.  
Moreover, for singular LQ systems, the constraints algorithm with partial feedback on the controls is equivalent to the presymplectic algorithm.
\end{theorem}

\begin{proof}
We will proof the equivalence of both algorithms by induction on $k$, the step of the algorithm.  For $k = 1$ is obvious as the expression of the primary constraints $\phi^{(1)}$, hence of $M_1$ is the same for both algorithms.   Let us consider the $k$th step of the algorithm.   Thus we have the presymplectic system $(M_k,\Omega_k, H_k)$ and the consistency condition eq. \eqref{consistencyk}.   If eq. \eqref{consistencyk} is satisfied at the point $(x,p,u)\in M_k$, then there will exists a vector $\Gamma$ in $M_k$ satisfying eq. \eqref{presymplectic_k}, thus because the constraints $\phi^{(k)}$ define $M_k$, then $\Gamma (\phi^{(k)}) = 0$ at $(x,p,u) \in M_k$.  Conversely, if eq. \eqref{recursive} is verified at some point $(x,p,u) \in M_k$, then there exist $C^a$ such that the vector in $M_k$ given by:
$$
\Gamma (x,p,u) = \frac{\partial H_k}{\partial p_i}\frac{\partial }{\partial x^i} - \frac{\partial H_k}{\partial x_i}\frac{\partial }{\partial p_i} + C^a\frac{\partial }{\partial u^a} 
$$
satisfies $\Gamma (\phi^{(k)}) = 0$.  However, this vector $\Gamma$ clearly satisfies eq. \eqref{presymplectic_k}, and then necessarily at the point $(x,p,u)$ it must be satisfied that $dH_k (Z) = 0$ for all $Z \in K_k$.

We will prove now the equivalence of the constraints algorithm with partial feedback on the controls and the previous two algorithms by induction on $k$.  At the first step of the algorithm the spaces $M_1$ and $\widetilde{M}_1$ are linearly isomorphic by the map $\alpha_1$ (eq. \eqref{alpha1}).   It was shown in Section \ref{first_iteration} that the consistency conditions for the reduced constraints $\widetilde{\phi}^{(1)}$ with respect to the vector field $\Gamma^{(1)}$ were equivalent to the secondary constratints $\phi^{(2)}$ defining $M_2$ (eq. \eqref{secondary_feed}) hence the equivalence at step 1.

Now suppose that we are at the step $k$ of the algorithm, thus we have on one side the subspace $M_k$ defined by the $k$--ary constraints $\phi^{(k)}$ and the subspace $\widetilde{M}_k \subset M^{(k)}$ defined by the reduced $k$--ary constraints $\widetilde{\phi}^{(k)}$ and the map $\alpha_k \colon \widetilde{M}_k \to M_k$ given by eq. \eqref{alphak} is a linear isomorphism.   Now the constraints algorithm consistency condition, eq. \eqref{recursive}, for $\phi^{(k)}$ implies that there exist $C^a$ such that $\dot{\phi}^{(k)} = 0$ when the derivative is taken in the direction of the vector 
$$
\Gamma = \dot{x}\frac{\partial }{\partial x^T} - \dot{p}\frac{\partial }{\partial p^T} + C\frac{\partial }{\partial u^T} 
$$
this is:
$$ \phi^{k+1} = \dot{\phi}^{(k)} = \Gamma(\phi^{(k)}) = S^{(k)} \matrix{c}{ \dot{x} \\\hline \dot{p}} - R^{(k)}C .$$
However by using the SVD factorization of $R^{(k)}$ we see immediately that there exists a solution of the previous linear equation if and only if:
$$
0= S_c^{(k)}\matrix{c}{ \dot{x} \\\hline \dot{p}} = \dot{\widetilde{\phi}}^{(k)},
$$
but this consistency condition for the partial feedback constraints algorithm, eq. \eqref{consistency_reducedk}, defines both the partial feedback eq. \eqref{feed_controlsk} and the $k+1$--ary reduced constraints, eq. \eqref{k+1reduced}, this is $\widetilde{M}_{k+1}$.  Moreover the map $\alpha_{k+1}\colon \widetilde{M}_{k+1} \to M_{k+1}$ defined by $\alpha_{k+1}(x,p,u^{(k+1)}) = (x,p,u^{k})$ with $u^{(k)} = [\widetilde{u}^{(0)}; u^{(k+1)}]$ and $\widetilde{u}^{(k)}$ given by the partial feedback  eq. \eqref{feed_controlsk} is a linear isomorphism and the equivalence is proved.
\end{proof}

When the algorithm stabilizes, this is after the step $\kappa$ we have: $M_\kappa = M_{\kappa+1} = \cdots $, the structure of the reduced system will be that of a presymplectic
system $(M_\kappa, \Omega_\kappa, H_\kappa )$ where $\Omega_\kappa$ denotes the restriction of $\Omega_0$ to the final constraints subspace $M_\kappa$ and the same for $H_\kappa$.   The presymplectic form
$\Omega_\kappa$ will have a kernel $K_\kappa$.    If $K_\kappa = 0$, then the form $\Omega_\kappa$
would be symplectic and we would have obtained a
consistent and explicit Hamiltonian form for the PMP with a well defined optimal feedback law.  

On the other hand if $K_\kappa \neq 0$ then the reduced consistent quasilinear DAE:
\begin{equation}\label{dyn_infty}
i_{\Gamma_\kappa} \Omega_\kappa = d H_\kappa ,
\end{equation}
will be just presymplectic,
meaning by that the there would persists an ambiguity in the dynamical equations. Such ambiguity will
come from $K_\kappa$ because if we have any consistent determination of $\dot{u}^a = C^a(x,p,u)$
for our system, adding to it any vector $Z\in K_\kappa$ will also be a consistent determination for
the dynamics.    We will call such ambiguity in the determination of the dynamics
the residual ``gauge'' symmetry of the system because of the resemblance with the case of singular Lagrangians and gauge symmetries.  We will describe explicitly the dependence of the differential conditions for the reduced PMP in terms of the constraints $\phi^{(k)}$ in Section \ref{first_second} after an appropriate extension of the presymplectic algorithm is presented.

The need for an extension of the presymplectic algorithm comes not only to analyze the gauge ambiguity of the reduced PMP but also because the presymplectic character of the form $\Omega_0$ obscures the numerical algorithm due to the singular nature of the corresponding matrices.  
It would be desirable to translate the previous analysis into a nonsingular (this is symplectic) setting.   This can be done by using a natural regularization of (arbitrary)
presymplectic systems based in the so called coisotropic embedding theorem \cite{Go81}.  The coisotropic embedding theorem roughly states that any presymplectic manifold can be embbeded in a symplectic manifold in an essentially unique way.   
We will take advantage of this theorem to transform our presymplectic system, either (\ref{presymplectic}) or (\ref{presymplectic_local}), into a genuine
symplectic Hamiltonian system and then we will reproduce the recursive constraints analysis discussed in Section \ref{consistency} leading to $(M_\kappa, \Omega_\kappa, H_\kappa )$. 

%%%%%%%%%%%%%%%%%%%%%%%%%%%%%%%%%%%%
%%%%%%%%%%%%%%%%%%%%%%%%%%%%%%%%%%%%

\subsection{Hamiltonian regularization of the presymplectic PMP}\label{regularization}

Due to the particular structure of the presymplectic
systems we are dealing with it won't be needed to use the coisotropic embedding theorem 
in its more general abstract setting.
In fact we will construct explicitly the various spaces needed and the theorem reduces
to the fact that such choices are natural and unique in an appropriate sense.

We want to extend the presymplectic manifold $(M_0 ,\Omega_0)$ to a symplectic manifold
$(\widetilde{M}, \widetilde{\Omega} )$ as well as the Hamiltonian $H_0$ and the constraints.
The total space $\widetilde{M}$ is obtained adding the dual space of the characteristic distribution  $K_0$ of
$\Omega_0$ to the control degrees of freedom, i.e., $\widetilde{M} \cong T^*P\times C\times K_0^* =
T^*P \times T^*C = T^*(P\times C)$.  Notice that the extension $\widetilde{M}$ thus constructed is already a symplectic manifold with the canonical symplectic form $\widetilde{\Omega}$ of $T^*(P\times C)$. The dual coordinates to the controls that we have
introduced will be denoted as $v_a$, $a = 1, \ldots, m$.  Finally, the symplectic extension $\widetilde{\Omega}$ of the presymplectic form $\Omega_0$ has the canonical form:
$$ 
\widetilde{\Omega} = dx^i \wedge dp_i + du^a \wedge dv_a .
$$
Moreover an extension of Pontryagine's Hamiltonian to $\widetilde{M}$ is given by
\begin{equation}\label{Htilde}
\widetilde{H}(x,p,u,v) = H_0(x,p,u) + C^a (x,p,u) v_a ,
\end{equation}
with $C^a = C^a(x,p,u)$ arbitrary functions, because the original space $M_0$ is sitting inside $\widetilde{M}$ as the submanifold defined by the
constraints
\begin{equation}\label{coisotropic_coordinates}
 \phi^{(0)}_a (x,p,u,v) = v_a .
\end{equation}
The extended Hamiltonian $\widetilde{H}$ defines a Hamiltonian dynamics $\widetilde{\Gamma}$ given by
\begin{equation}
i_{\widetilde{\Gamma}}\widetilde{\Omega}=d\,\widetilde{H},
\end{equation}
and the vector field $\widetilde{\Gamma}$ has the expression:
\begin{eqnarray}\label{Gammatilde}
\nonumber\widetilde{\Gamma} &=&
\frac{\partial\widetilde{H}}{\partial p_i}\frac{\partial}{\partial x^i}
-\frac{\partial\widetilde{H}}{\partial x^i}\frac{\partial}{\partial p_i}
+\frac{\partial\widetilde{H}}{\partial v_a}\frac{\partial}{\partial u^a}
-\frac{\partial\widetilde{H}}{\partial u^a}\frac{\partial}{\partial v_a}=\\
&=&\frac{\partial \widetilde{H}}{\partial p_i}\frac{\partial}{\partial x^i} -\frac{\partial \widetilde{H}}{\partial
x^i}\frac{\partial}{\partial p_i} +C^a\frac{\partial}{\partial u^a} -\frac{\partial \widetilde{H}}{\partial
u^a}\frac{\partial}{\partial v_a}. \end{eqnarray}
Since we have extended the space, we must impose some constraints to recover the original problem. These constraints are given, as it was said before, by the
coisotropic coordinates $v_a$, eq. (\ref{coisotropic_coordinates}), i.e.,
\begin{equation}\label{extendedM0}
M_0=\{(x,p,u,v)\in\widetilde{M}|\,\phi^{(0)}_a = v_a=0\} .
\end{equation}
These new constraints $\phi^{(0)}_a$ will be called zero order constraints (see Table \ref{summary} for a summary of the two pictures of the PMP: presymplectic and regularized Hamiltonian).

The vector field  $\widetilde{\Gamma}$ applied to the
zero order constraints $\phi^{(0)}_a = v_a$ give the primary constraints that we had before, eq. \eqref{liga}, $$
\widetilde{\Gamma}\,(v_a)=-\partial H/\partial u^a = -\phi^{(1)}.
$$ 
Thus, $\widetilde{\Gamma}$ is tangent to $M_0$ if and only if $\partial H/\partial u^a=0$, as we have established previously.

%\begin{center}
\begin{table}[h!]
\caption{The presymplectic and regularized Hamiltonian pictures of PMP for singular optimal control problems.}\label{summary}
\begin{tabular}{@{}l l|c|c|}\\
\cline{3-4}&&&\\&&Presymplectic& Regularized Hamiltonian\\\hline\vline&&&\\
\vline& Phase space &$M_0=T^*P\times C$&$M_{-1} = \widetilde{M}= T^*(P\times C)$\\\vline&&&\\
\hline\vline&&&\\\vline&2-form&$\Omega_0=dx^i\wedge dp_i$& $\widetilde{\Omega}=dx^i\wedge
dp_i+du^a\wedge dv_a$\\\vline&&&
\\\hline\vline&&&\\\vline&Hamiltonian&$H_0$&$\widetilde{H} = H_0 + C^a v_a$\\
\vline&  &&\\\hline \vline&&&
\\\vline&Constraints ($k\geq 1$)&$\phi^{(k)}=\phi^{(k)}(x,p,u)$&$ 
\phi^{(k)} + \sum_{l=1}^k \lambda^b_{al}\,\phi^{(k-l)}_b ,$\\\vline&&&\\
\vline&Constraint ($k=0$)&&$ \phi^{(0)}_{a}=v_a$
\\\vline&&&\\\hline\\
\end{tabular}
\end{table}
%\end{center}

Now,  applying the constraints algorithm to the extended space and denoting $\widetilde{M}$ by $M_{-1}$ for consistency with the previous notation, we have that eq. \eqref{extendedM0} reads 
$M_0=\{(x,p,u,v)\in M_{-1} |\,\phi^{(0)}_a =0\}$.   Denoting by $i_0\colon M_0 \to M_{-1}$ the canonical embedding, we have $i_0^*\widetilde{\Omega} = \Omega_0$ and $i_0^*\widetilde{H} = H_0$.  Thus the presymplectic system $(M_0,\Omega_0,H_0)$ emerges as the result of applying the presymplectic algorithm to $(\widetilde{M}, \widetilde{\Omega}, \widetilde{H})$ with respect to the zero order constraints $\phi^{(0)}$.   From there we iterate the presymplectic algorithm obtaining the results described in the previous Section.  Thus we have:

\begin{theorem}\label{equivalence}
Given the presymplectic system $(M_0,\Omega_0,H_0)$ corresponding to a singular optimal control problem, then the recursive presymplectic algorithm of this system and the constraints algorithm of the regularized Hamiltonian system $(\widetilde{M},\widetilde{\Omega},\widetilde{H})$ with respect to the constraint submanifold $M_0 \subset \widetilde{M}$ are equivalent.
\end{theorem}

\begin{proof}
Consider the submanifold $\widetilde{M}_0$ of $\widetilde{M}$.  If we impose
$\G_{\widetilde{H}}$ to be tangent to it, this happens only at the points  $\xi\in \widetilde{M}$
such that $Z(\widetilde{H})=0$, where $Z\in K=\ker\,\widetilde{\Omega}\mid_{M_0} = \Omega_0$.
But $\ker\Omega_0$ is spanned by the vectors $\partial/\partial u^a$ wich are the Hamiltonian vector fields corresponding to the coisotropic coordinates, this is $i_{\partial/\partial u^a} \widetilde{\Omega} = dv_a$.  Thus $\G_{\widetilde{H}}$ will be tangent to
$M_0$ if $\G_{\widetilde{H}}(v_a)= \partial \widetilde{H}/\partial u^a$ (see eq. \eqref{Gammatilde}), but
$\widetilde{H}=H_0+C^av_a$, then $\G_{\widetilde{H}}(v_a)= \p{H_0}/\p{u^a} + v_b \p{C^b}/\p{u^a}$,
thus on $M_0$, we have $\G_{\widetilde{H}}(v_a)=0$ if and only if $\p{H_0}/\p{u^a}=0$.  Hence the next step of the algorithm will agree with the first step of the recursive presymplectic algorithm (see Secction \ref{recursive_presymplectic}) and the equivalence is proved.
\end{proof}

See the Table \ref{summary} for a summary of the two
pictures of the PMP: presymplectic and regularized Hamiltonian.
Notice that any extension of Pontryagine's Hamiltonian and the constraints would be valid because they would lead to the same final constraint submanifold $M_{\kappa}\subset \cdots \subset M_1\subset M_0\subset M_{-1}$.  We have chosen the extension eq. \eqref{Htilde} for convenience because if provides the simplest setting for the numerical implementation of the algorithm.

%%%%%%%%%%%%%%%%%%%%%%%%%%%%%%%%%%%%
%%%%%%%%%%%%%%%%%%%%%%%%%%%%%%%%%%%%

\subsection{The Hamiltonian nature of the constraints algorithm with partial feedback on the controls}

In this section we will show that the constraints algorithm with partial feedback on the controls is Hamiltonian, i.e., we prove that at each iteration the vector field $\Gamma^{(k)}$ is Hamiltonian, hence the differential conditions describing the reduced PMP obtained by using the algorithm are Hamiltonian.  

We have already shown that the presymplectic constraints algorithm given by eq. \eqref{recursive_sym}, the constraints algorithm eq. \eqref{recursive} and the constraints algorithm with partial feedback on the controls, are equivalent (Thm. \ref{equivalence_constraints}).   However the equivalence was stated at the level of the subspaces determined by the constraints (and the constraints themselves) of each one of the algorithms.  The dynamics (or differential conditions) for each one where determined in a different way.  In the constraints algorithm, the dynamics was determined by the consistency condition itself that precisely selected the points such that a vector that preserves the constraints and that belongs to the initial dynamics exists.  The presymplectic algorithm determines the dynamics by means of the dynamical equation eq. \eqref{presymplectic_k} at each iteration (and is explicitly Hamiltonian because of this), and the constraints algorithm with partial feedback on the controls determines the dynamics at each step by solving part of the controls (eq. \eqref{gamma_k}).   Because the subspaces $M_k$ defined by the presymplectic and the constraints algorithm are the same, the presymplectic forms of both coincide and because the dynamics defined by the constraints algorithm is a particular instance of the original differential condition that satisfies the presymplectic equation, the induced dynamics will be Hamiltonian too.   However the partial feedback algorithm defines the dynamics in a different way but it is Hamiltonian too as it is shown in the following proposition:

\begin{theorem}\label{ham_feed}
At each iteration of the constraints algorithm with partial feedback on the controls it is satisfied:
$$
J \Gamma^{(k)} = \nabla H^{(k)} ,
$$
where $H^{(k)}$ denotes Pontryagine's Hamiltonian where the partial feedbacks has been applied to the controls $\tilde{u}^{(0)}, \ldots, \tilde{u}^{(k-1)}$ and $\nabla H^{(k)}$ is the column vector  $\nabla H^{(k)} =[\partial H^{(k)}/\partial x^T ; \partial H^{(k)}/\partial p^T ]$.
\end{theorem}

\begin{proof}  We will prove it by induction.  At the step zero, it is trivial to show that:
$$
J \Gamma^{(0)} = \nabla H^{(0)}
$$
where 
$$
\Gamma^{(0)} = \matrix{c}{\dot{x} \\ \dot{p}} = \matrix{c}{Ax + Bu \\ Qx - A^Tp + Nu}
$$
and $H^{(0)} := H_0$.

As stated in the proposition, $H^{(k)}$ is obtained from $H^{(k-1)}$ by applying the partial feedback to the controls $\tilde{u}^{k-1}$, hence by restricting $H^{(k-1)}$ to the subspace $\widetilde{M}_k \subset \widetilde{M}_{k-1}$.  Taking the differential commutes with the restriction map, thus if we apply the restriction map to the equation $ J \Gamma^{(k-1)} = \nabla H^{(k-1)}$ on both sides, we will get $J \Gamma^{(k)} = \nabla H^{(k)} $ because the restriction of $\Gamma^{(k-1)}$ is exactly $\Gamma^{(k)}$ (recall that $\Gamma^{(k)}$ is defined just applying the partial feedback to $\Gamma^{(k-1)}$).
\end{proof}

In what follows we will simply call the ``Hamiltonian constraints algorithm'' to the recursive constraints algorithm with partial feedback on the controls applied to the extended space $M_{-1}$ with zero order constraints $\phi^{(0)}_a = v_a$, and no mention will be made to the partial feedback on the controls unless it is needed.

Diagrame \ref{diagramme} summarizes the various algorithms with all the maps and spaces presented so far.  The left--hand side of the diagrame shows the presymplectic constraints algorithm and the right--hand side the partial feedback on the controls.  Notice that the first row consists on the symplectic space $T^*{P\times C}$ that is denoted either by $M_{-1}$ or by $M^{(-1)}$ depending if we want to emphasize the starting point of the presymplectic or the partial feedback algorithm respectively.    The spaces on the second row are again all the same, the total space $M_0 = T^*P \times C$ of the problem, but again, on the left--hand side it appears as the result of applying the zero order constraints $\phi^{(0)} = 0$ to $M_{-1}$, while in the right--hand side it appears as the result of using the partial feedback $v_a = 0$ to the new ``control'' variables.     

The column on the left shows the sequence of spaces $M_k = \{ \phi^{(k)} = 0\}$ and natural inclusions $i_k\colon M_k \to M_{k-1}$ obtained by applying the recursive constraint algorithm eq. \ref{recursive}.  The central column shows the sequence of spaces $\widetilde{M}_k= \{ \widetilde{\phi}^{(k)} = 0\}$ and inclusions $\tilde{i}_k \colon \widetilde{M}_k \to \widetilde{M}_{k-1}$ obtained by applying the partial feedback algorithm eq. \eqref{Mtilde_k} as well as the isomorphisms $\alpha_k$ between  $\widetilde{M}_k$ and $M_k$, eq. \eqref{alphak}.  

Finally, the column on the right shows the sequence of spaces $M^{(k)} = \mathbb{R}^{2n} \times \mathbb{R}^{m_k}$ with the controls $u^{(k)}$ left after the previous partial feedbacks and the corresponding inclusions $j_k\colon M^{(k)} \to M^{(k-1)}$.  The bottom row shows the final constraints space $M_\kappa$ that describes the reduced PMP, or equivalently the space $\widetilde{M}_\kappa$ which is contained in $M^{(\kappa)}$ which is the space whose points are $(x,p,u^{(\kappa)})$, with $u^{(\kappa)}$ the set of controls that remain free after the reduction process.   

%$$\begin{array}{cc}
\begin{figure}
\begin{center}
$$\xymatrix{ 
%& M\ar[dr]_{rp_P}\ar[r]^{X} & T(T^*P)\ar[d]^{\tau_{T^*P}} \\ 
%& &T^*P& \\
\mathrm{Constraints} & M_{-1} = \mathbb{R}^{2n}\times \mathbb{R}^{2m_0} = M^{(-1)} & \mathrm{Partial \,\, feedback} \\
M_0\ar[ur]^{\phi^{(0)}=0}_{i_0} & \mathbb{R}^{2n}\times \mathbb{R}^{m_0}\ar[u]\ar[r]\ar[l] & M^{(0)}\ar[ul]^{j_0}_{\quad\quad  \widetilde{u}^{(-1)} = v^{(0)} = 0}  \\
M_1\ar[u]^{\phi^{(1)}=0}_{i_1} & \widetilde{M}_1\ar[u]^{\widetilde{\phi}^{(1)}=0}_{\tilde{i}_1}\ar[r]^{\tilde{j}_1}\ar[l]_{\alpha_1} & M^{(1)}\ar[u]^{j_1}_{\,\, \widetilde{u}^{(0)} = \feed^{(1)} \, \zeta} \\  
M_2\ar[u]^{\phi^{(2)}=0}_{i_2} & \widetilde{M}_2\ar[u]^{\widetilde{\phi}^{(2)}=0}_{\tilde{i}_2}\ar[r]^{\tilde{j}_2}\ar[l]_{\alpha_2}  & M^{(2)}\ar[u]^{j_2}_{\,\, \widetilde{u}^{(1)} = \feed^{(2)} \, \zeta}  \\ 
\vdots\ar[u] & \vdots\ar[u] & \vdots\ar[u] \\
M_{\kappa-1}\ar[u] & \widetilde{M}_{\kappa-1}\ar[u]\ar[l]\ar[r] & M^{(\kappa-1)}\ar[u] \\
M_\kappa\ar[u]^{\phi^{(\kappa)}=0}_{i_\kappa} & \widetilde{M}_\kappa\ar[u]^{\widetilde{\phi}^{(\kappa)}=0}_{\tilde{i}_\kappa}\ar[r]^{\tilde{j}_\kappa}\ar[l]_{\alpha_\kappa}  & M^{(\kappa)}\ar[u]^{j_\kappa}_{\,\, \widetilde{u}^{(\kappa-1)} = \feed^{(\kappa)} \, \zeta} 
}
$$
\caption{The recursive Hamiltonian and the partial feedback constraints algorithms} \label{diagramme}
\end{center}
\end{figure}
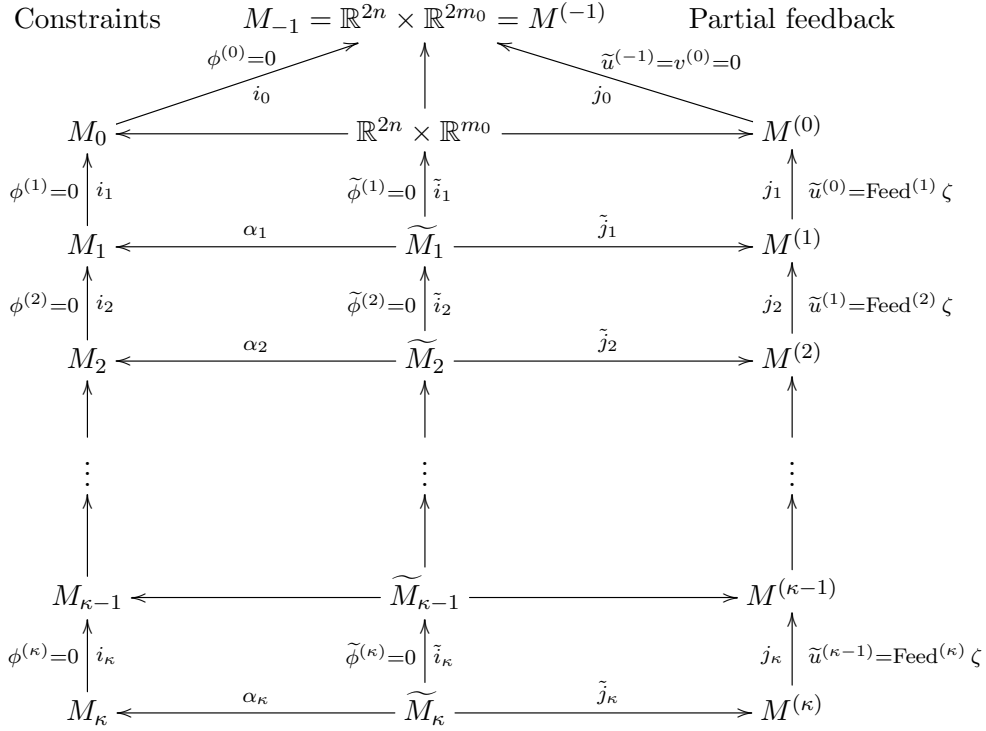
%& {\bfig\xymatrix{ &M\ar[dr]\ar[r]^{\rp_P} & T^*P\ar[d]^{\pi_P}\\   && P&} \efig} \end{array} $$    

%%%%%%%%%%%%%%%%%%%%%%%%%%%%%%%%%%%%%%%%%%%%
%%%%%%%%%%%%%%%%%%%%%%%%%%%%%%%%%%%%%%%%%%%%

\subsection{First and second class constraints and the structure of the reduced PMP}\label{first_second}

Once obtained the reduced presymplectic system $(M_{\kappa}, \Omega_\kappa, H_\kappa)$ we analyze the family of vector fields $\Gamma_\kappa$ satisfying the dynamical equation eq. \eqref{dyn_infty}.

First we  classify the set of constraints $\{ \phi^{(k)}_a \mid k = 0,1,2,\ldots \}$ as first or second class constraints accordingly with the characteristics of the Hamiltonian vector fields defined by them. 
Again the terminology is reminiscent of the theory of constraints for singular Lagrangian systems (see for instance \cite{Sn74} and references therein).  Such decomposition of the set of constraints will
lead us to a neat characterization of the structure of the presymplectic manifold $(M_\kappa, \Omega_\kappa)$ being symplectic if all constraints are second class, while it will possess a nontrivial kernel $K_\kappa \neq 0$ if and only if there are non--vanishing first class constraints.  

Consider $(\widetilde{M},\Omega)$ a symplectic manifold and a
regular submanifold $i\colon N\hookrightarrow \widetilde{M}$. 
We will say that a function $\phi$ vanishing on $N$ is a first class constraint if 
for any function $\psi$ vanishing on $N$, $\psi|_N = 0$, the Poisson bracket of both functions vanish on $N$, i.e., $\{ \phi, \psi \}|_N = 0$. 
Otherwise case we will 
call $\phi$ a second class constraint. 

It is well--known that given a family of independent constraints $\Phi = \{ \phi_j \mid j = 1, \ldots p \}$,  we can always decompose it
 into two subfamilies $\Phi = \set{\nu_A,\, \chi_\alpha \mid A = 1, \ldots, r, \, \alpha = 1, \ldots ,2s}$, $r + 2s = p$, where $\nu_A$
and $\chi_\alpha$ are first and second class constraints respectively and they verify that:
\begin{eqnarray} \label{comm_relations}
\{\nu_A,\nu_B\} & = & C_{AB}^C\nu_{C},\\ \nonumber
\{\nu_A,\chi_\alpha\} & = & C_{A\alpha}^B\nu_{B},\\ \nonumber
\{\chi_\alpha,\chi_\beta\} & = & C_{\alpha \beta}
\end{eqnarray}
with $C_{\alpha \beta}$ is non-degenerate (see more details in \cite{Sn74}).

\begin{theorem}  Given the final reduced presymplectic system $(M_\kappa, \Omega_\kappa, H_\kappa)$ associated to a singular optimal control problem $K _\kappa = \ker \Omega_\kappa$ is generated by the Hamiltonian vector fields $\G_{\nu_A}$ corresponding to the first class constraints $\nu_A$.   The vector fields $\Gamma_\kappa$ whose integral curves are the optimal extremals of the problem, have the form:
\begin{equation}\label{reduced_vector}
\G_{\kappa} = \G_{H_\kappa} +\sum_{A= 1}^r\l^A\G_{\nu_A}.
\end{equation}
where $\nu_A$, $A = 1, \ldots, r$, are the first class constraints of the system.
\end{theorem}

\begin{proof}  It is a general fact that Hamiltonian vector fields corresponding to first class constraints generate the characteristic distribution of the presymplectic form $\Omega_\kappa$ (see for instance \cite{Sn74}).

On the other hand, the most general hamiltonian in $M_\kappa$ extending $H_\kappa$ is
$$
H= H_\kappa + \l^A\nu_A + \l^\alpha\chi_\alpha ,
$$
and the vector field associated  to it is
$$
\G = \G_{H_\kappa} + \l^A\G_{\nu_A} + \l^\alpha \G_{\chi_\alpha} + \nu_A\G_{\l^A} + \chi_\alpha\G_{\l^\alpha}.
$$
Thus, restricted to $M_\kappa$ this vector field has the form
$$\left.\G\right|_{M_{\kappa}}=\G_{H_\kappa}+\l^A\G_{\nu_A}+\l^\alpha\G_{\chi_a}.$$
Since this vector field must leave invariant the set of constraints because it must be tangent to $M_\kappa$, then:
$$
\G(\nu_B) = \G(\chi_b)=0,
$$
on $M_\kappa$.  But $\G_{H_\kappa}$ already leaves them invariant by construction, then we have $0= \Gamma (\chi_\beta )\mid_{M_\kappa} = \{ \l^A\nu_A + \l^\alpha\chi_\alpha, \chi_\beta \}\mid_{M_\kappa} = \lambda^\beta  C_{\alpha \beta}$ and consequently $\l^\alpha=0$.
\end{proof}

Hence the reduced PMP $(M_\kappa, \Omega_\kappa, H_\kappa)$ of a singular optimal control problem could be either:
\begin{enumerate}

\item Nondegenerate: $M_\kappa$ is a symplectic submanifold of $\widetilde{M}=M_{-1}$ with $\Omega_\kappa$ a closed and non--degenerate
2--form and a hamiltonian $H_\kappa$.  The vector field $\Gamma_\kappa$ eq. \eqref{reduced_vector}, whose integral curves are optimal extremals for the problem is uniquely determined and there are no undetermined controls.  There must be an optimal feedback law determining the controls (even though some of them as well as some variables $x$'s and/or $p$'s could have been removed along the reduction procedure).

\item Degenerate:  $M_\kappa$ is a  presymplectic submanifold of $\widetilde{M} = M_{-1}$. The 2--form is presymplectic and its kernel $K_\kappa$ is generated by first class constraints.   The differential conditions implementing the reduced PMP are undetermined and there will be a ``gauge'' freedom determined by the first class constraints of the system with the form given by eq. \eqref{reduced_vector} with the $\lambda^
A$ being the ``effective controls'' of the theory.  The effective controls $\lambda^A$ could be either some of the original controls $u^a$ or new functions of the variables $x,p,u$.
\end{enumerate}

In the following sections we will construct a numerical algorithm that will provide not only
$M_\kappa$ but the decomposition of the family of constraints into first and second class and the explicit form of $\Gamma_\kappa$.

%%%%%%%%%%%%%%%%%%%%%%%%%%%%%%%%%%%%%%%%%%%%%
%%%%%%%%%%%%%%%%%%%%%%%%%%%%%%%%%%%%%%%%%%%%%

\section{The numerical Hamiltonian constraints algorithm}

%%%%%%%%%%%%%%%%%%%%%%%%%%%%%%%%%%%%
%%%%%%%%%%%%%%%%%%%%%%%%%%%%%%%%%%%%

\subsection{The Hamiltonian constraints algorithm and first and second class constraints}

In this section we will complete the numerical implementation of the Hamiltonian constraints algorithm with partial feedback on the controls in the setting of singular LQ systems.   After the discussion in the previous Section, what is left to conclude the implementation of the algorithm is the analysis of the computation of first and second class constraints.

If we denote by $\{ \phi_i \mid i = 1, \ldots, p \}$ the family of constraints describing the submanifold
$M_\kappa \subset \widetilde{M}$ (notice that $\dim M_\kappa = 2n + 2m - p$), we form the matrix $\pois = [\pois_{ij}]$ whose entries are given by:
\begin{equation}\label{poisson_matrix}
\pois_{ij} = \{ \phi_i, \phi_j \} .
\end{equation}
If we restrict the values of its entries to
$M_\kappa$, it is clear because of eqs. (\ref{comm_relations}) that the kernel of the corresponding matrix will be spanned by the
rows corresponding to brackets with first class constraints.  In fact we obtain:

\begin{theorem}\label{separacionlig-c3}
Let $\{\phi_i \}$, $i=1,\ldots,p$, be an independent
set of functions defining a a submanifold $N$ of a symplectic manifold $\widetilde{M}$ and let
these constraints be such that they are linear functions in local canonical coordinates of the symplectic manifold $\widetilde{M}$. Denote by
$\pois_{ij}=\{\phi_i ,\phi_j \}$, $i,j=1,\ldots,p$, the $p\times p$ constant matrix built as the Poisson
brackets of the constraints.   Let  $v^{(A)}$ be a basis of the kernel of
the matrix $\pois$, $A = 1,\ldots, r$, where  $2s = p-r$ is its rank.  Let $w^{(\alpha)}$,
$\alpha=1,\ldots,2s$ be a family of independent vectors completing the basis $v^{(A)}$. 
Then, the functions
\begin{equation}\label{1st-c3}
\nu_A = \sum_{j = 1}^p \phi_j \, v_j^{(A)}, \quad A=1,\ldots,r,
\end{equation}
and
\begin{equation}\label{2nd-c3}
\chi_\alpha = \sum_{j = 1}^p \phi_j \, w_j^{(\alpha)}, \quad  \alpha=1,\ldots,2s,
\end{equation}
are basis of the first and second class constraints respectively.
\end{theorem}

\begin{proof}   First, we notice that if $\phi_j$ are linear functions in local canonical coordinates $(\tilde{q}^a, \tilde{p}_a)$ of the symplectic manifold $\widetilde{M}$, then the Poisson brackets 
$$
\pois_{jk} = \{ \phi_j,\phi_k \} = \sum_a \frac{\partial \phi_j}{\partial \tilde{q}^a} \frac{\partial \phi_k}{\partial \tilde{p}_a} -   \frac{\partial \phi_j}{\partial \tilde{p}_a} \frac{\partial \phi_k}{\partial \tilde{q}^a} ,
$$
are constants.
The $p\times p$ matrix $\pois$ is skew--symmetric, thus its rank must be even, say $2s$, hence its kernel will have dimension $r = p - 2s$.

Now, the constraints $\nu_A$ and $\chi_\alpha$ defined by eqs. \eqref{1st-c3} and \eqref{2nd-c3} are a  linear combination of the starting constraints.
When computing the Poisson brackets among themselves we get: 
\begin{eqnarray*}
\{\nu_A ,\nu_B \} & = &
\{\sum_j \phi_j \, v_j^{(A)}, \sum_k \phi_k \,v_k^{(B)}\} =
\sum_{j} v_j^{(a)} \left(\sum_k \{\phi_j , \phi_k \} v_k^{(B)} \right) \\ & = &
\sum_{j,k} v_j^{(a)} \left( \sum_k \pois_{jk} v_k^{(B)} \right)= \sum_j v_j^{(a)} \left( \pois \cdot v^{(B)}\right)_j = 0,\\
\{\chi_\alpha , \nu_A \} & = & \sum_{k,j} \{\phi_k \, w_k^{(\alpha)}, \phi_j \,v_j^{(A)}\} =
\sum_k w_k^{(\alpha)} \left(\pois_{kj} v_j^{(A)}\right) =0.
\end{eqnarray*}

From these expressions we
conclude that the constraints $\nu_A$ are of first class, because all the  brackets of them
with the remaining  constraints vanish on $N$.  Finally, the brackets of the constraints defined in
(\ref{2nd-c3}) with themselves define a nondegenerate matrix (of rank $2s$): 
\begin{eqnarray*}
C_{\alpha \beta} = \{\chi_\alpha ,\chi_\beta \} =\sum_{j,k} \{\phi_j \, w_j^{(\alpha)}, \phi_k \,w_k^{(\beta)}\} =
\sum_{j,k}w_j^{(\alpha)}\{\phi_j, \phi_k \} w_k^{(\beta)} = \pois (w^{(\alpha)}, w^{(\beta)}) ,
\end{eqnarray*}
because by construction the matrix $\pois$ is non--degenerate when
restricted to the  subspace $W = \mathrm{span} \{ w^{(\alpha)} \mid \alpha = 1,\ldots, r\} $, then the matrix $C_{\alpha\beta}$ is invertible.
\end{proof}

We apply the previous results to the regularized Hamiltonian system on $\widetilde{M} = T^*(P\times C)$.  The total coordinates of the extended symplectic linear space $\R^{2n}\times\R^{2m}$
problem are $(x,p\,;u,v)$ where we have added the coisotropic coordinates, $v_a$, as column vectors. 

In matrix form the set of constraints $\{  \phi_j \mid j = 1, \ldots p\}$ is written as: 
$$\Phi [x,p,u,v]' = 0,$$ 
where the rows of the $p\times (2n+2m)$ matrix $\Phi$ are the coefficients of the constraints, this is the $j$th row of $\Phi$ will represent the $j$th linear constraint $\phi_j$ with general form:
$$
\Phi (j,:) = \sigma(j,:)*x(:) + \beta(j,:)*p(:) + \rho(j,:)*u(:) + \omega(j,:)*v(:)
$$
and, from now on, we just call $\Phi$ the set
of constraints. 

Because of the analysis of the regularized Hamiltonian problem in Section \ref{regularization} and Thm. \ref{equivalence},  to obtain the reduced PMP starting from the symplectic manifold $\widetilde{M}$ we must start with the zero order constraints $\phi^{(0)}_a = v_a$,
forcing these coordinates to vanish.  In matrix form they are the first $m$ rows of $\Phi$ and the second set of $m$ rows are given by the set of primary constraints $\phi^{(1)}_a$, $a = 1,\ldots, m$.  Thus the matrix $\Phi$ begins as follows:
$$\Phi = \matrix{cccc}{0_{m\times n}&0_{m\times n}&0_{m\times m}&-I_m\\
&&&\\\sigma^{(1)}&\beta^{(1)}&\rho^{(1)}&0_{m\times m}}.
$$ 
with $[\sigma^{(1)} \mid \beta^{(1)}] = S^{(1)}$ (eq. \eqref{S1}).

\subsection{Computing the first and second class constraints}
Notice that if we have a matrix and we add rows or columns the rank can only increase or remain unchanged and any family of previously independent rows will keep being independent.   Thus, at any step of the algorithm, when we add new constraints to $\Phi$ and we form the new matrix $\poi$, the family of second class constraints can only increase or remain unchanged.  Thus we can separate at each step the second class constraints by just computing the Poisson brackets of the new constraints with the previous first class constraints and analyzing the kernel of the corresponding matrix. 

Thus at each step $k$ of the iteration we separate the new second class
constraints.  

\noindent Denoting the constraints obtained in the step  $k$ as:\\
\begin{tabular}{ll}
$\Phi_I(k):$ & set of first class constraints obtained until step $k$\\
$\Phi_{II}(k):$ & set of second class constraints  obtained until step $k$\\
$\pois(k):$ & brackets' matrix  of $\{\Phi_I(k),\Phi_I(k)\} = 0$ \\
$F$:&  new set of constraints.
\end{tabular}

We update the matrix $F$ adding the new set of constraints obtained to the first
class ones: 
$$
\widetilde{F} = \matrix{c}{\Phi_I(k)\\ F}.
$$ 
We compute the Poisson brackets matrix $\pois(k+1)$ computing the Poisson brackets among the rows of $\widetilde{F}$, this is:
$$
\pois(k+1)=\matrix{c|c}{ 0 &\{\Phi_I(k), F\}\\\hline\{F, \Phi_I(k)\}&\{F,F\}}.
$$

We compute a basis $v^{(A)}$ of the kernel of this matrix and we form the matrix: $v=\matrix{c|c|c}{v^{(1)}&\cdots&v^{(r)}}$ and another matrix whose
columns complete the basis  $v^{(A)}$.  We denote this matrix as
$w=\matrix{c|c|c}{w^{(1)}&\cdots&w^{(2s)}}.$ 

The set of first class constraints, eq. (\ref{1st-c3}), is given 
\eqy{\Phi_I(k+1)=v^T*\widetilde{F},}
and the second class constraints obtained at this step, eq. (\ref{2nd-c3}), is given by:
\eq{\widetilde{\Phi}_{II} = w^T*\widetilde{F} .}
Now, we add them to those obtained in previous steps to form the set of second class constraints until
step $k$, that is,
\eqy{\Phi_{II}(k+1)=\matrix{c}{\Phi_{II}(k)\\\\\widetilde{\Phi}_{II}}.}

Finally we will discuss briefly the reduction of columns suffered by all the constraints at each step of the algorithm due to the partial feedback.  Consider a constraint $\phi$ given as $\phi (x, p, u, v ) = \sigma x+ \beta p+ \rho u+ \omega v = 0$.
The reduction of columns due to the feedback is as follows. 
After obtaining the feedback of $u^{(0)}$ in
the first step, we have, eq. \eqref{u0}, 
$u^{(0)} =V^{(1)}\matrix{c}{\tilde{u}^{(0)}(1:r_1)\\\tilde{u}^{(1)}}$, 
since $\tilde{u}^{(0)}(1:r_1) = \feed^{(1)}\matrix{c}{x\\p}$, we perform the following simultaneous change:
\begin{eqnarray*}
\rho u + \omega v &=& [ \rho V^{(1)}](:,1:r_1)\feed^{(1)}\matrix{c}{x\\p} + \\
 &+& [\rho V^{(1)}](:,r_1+1:m_1)u^{(1)} +[\omega V^{(1)}](:,r_1+1:m_1)) v^{(1)},
 \end{eqnarray*}
where we
have used the fact that   the zeroth constraint $\{v_a\}=0$ must vanish, so we can get
rid of the part $[\omega V^{(1)}]\tilde{v}^{(0)}$. This procedure will be repeated at each step of the
algorithm and at the end we will remove the remaining coisotropic variables. 

%%%%%%%%%%%%%%%%%%%%%%%%%%%%%%%%
%%%%%%%%%%%%%%%%%%%%%%%%%%%%%%%%

%\newpage

\section{Implementation of the algorithm}
In this section we describe the pseudocode for the numerical regularized Hamiltonian constraints algorithm. 

%\bigskip

\noindent\begin{minipage}[c]{\textwidth}\vspace{0.3cm}

\hrule~\\ \textbf{Scheme 1:  \emph{HCAPF ``Hamiltonian constraints algorithm with partial feedback on the controls''}} (Regularized Hamiltonian recursive constraints algorithm splitting first and second class constraints for singular optimal LQ control problems) \vspace{0.1cm}\hrule

\bigskip

\begin{tabbing}
%Initialize the variables: $\sigma, \beta, \rho$; $\feed$ and $\nofeed$.\\
Initialize the variables: $S, R$; $\feed$ and $\nofeed$.\\
Build the constraints matrix $\Phi$ adding the coisotropic variables\\
Compute the Poisson brackets matrix $\pois$ of $\Phi$ \\
Split $\Phi$ in first and second class constraints \\\\

{\bf Regular problems} \\
Compute the total feedback and final vector field (there are no constraints) \\\\
{\bf Singular} \= {\bf problems}\\
  \>\textbf{while} \= the number of independent constraints increases\\
  \>\>\textbf{if} \= rank$(R)>0$\\
        \>\>\> Compute the partial feedback: $\feed$, $\feedtot$ and $\nofeed$\\
        \>\>\> Compute  the  vector field: $G$ and $Z$ using the partial feedback\\
        \>\>\> Include the feedback into the first and second class constraints\\
        \>\>\> Iterate the matrices for next step\\
     \>\>\textbf{endif} \\
   \>\> Add the new set of constraints to $\Phi_I$\\
   \>\> Compute the new Poisson bracket matrix $\pois$\\
   \>\> Split the set of constraints in first and second class\\
   \>\> Add the new second class constraints to the previous ones\\
   \>\textbf{endwhile}\\
  
Remove the coisotropic variables\\\\
Obtain the final vector field, total feedback, first and second class constraints\\
\end{tabbing}\hrule
\end{minipage}

\medskip

%\newpage

%%%%%%%%%%%%%%%%%%%%%%%%%%%%%%%%%%%%%%%%%%%%%%

\noindent\begin{minipage}[c]{\textwidth}\vspace{0.1cm} \hrule~\\ \textbf{Pseudocode 1:
\emph{HCAF ``Hamiltonian constraints algorithm with partial feedback on the controls''}} \vspace{0.01cm}\hrule

\begin{tabbing}\\ 
\textbf{input} $A,~B,~Q,~N,~R,~tol$\\
\com{Inicialization of variables:}\\
$S \leftarrow\matrix{c}{-N^T,B^T};~R \leftarrow-R$;  $r\leftarrow$rank$(R,tol);$
$[l,m]\leftarrow$size($R$); \\
$\feedtot=\matrix{c}{\,};$ $\feed = \matrix{c}{\,}; \,$\com{Feedback}\\
$G\leftarrow\matrix{cc}{A&0_{n\times n}\\Q&-A^T}$; $Z\leftarrow\matrix{c}{B\\N}$;
$[n,n]\leftarrow$size($A$);\com{Vector Field}\\
$\rfeed\leftarrow0$;  $M\leftarrow m$; $k\leftarrow0;$ \\
\com{Addition of the coisotropic constraints.}\\
$[S\mid R]\leftarrow $ \verb+independent_rows+ $([S\mid R],tol)$;~
$ \mathrm{Phi} \leftarrow\matrix{ccc}{0_{m\times 2n}&0_{m\times m}&I_m\\
 &&\\ S & R &0_{m\times m}}$ ;\\

$\pois\leftarrow$ \verb+poisson_brackets+ $(\mathrm{Phi},n,m,tol)$; \com{bracket of the constraints.}\\
$[v,w]\leftarrow$ \verb+numerical_ker+ $(\pois,tol)$;\\
$\mathrm{Phi}_I\leftarrow v^T \mathrm{Phi};~~ \mathrm{Phi}_{II}\leftarrow w^T \mathrm{Phi}$; \com{1st and 2nd class constraints.}\\\\

\com{\bf Regular problems}\\
\textbf{if} rank\=$(R,tol)==m$ \\
$[U^T,R,V^T]\leftarrow$svd$(R)$; $\rfeed\leftarrow m$;\\
     \> $\feed \leftarrow R^{-1} * U * S;$  $\feedtot=V;$\\
     \>  $G\leftarrow G+ Z * V^T$;  $\nofeed,~\mathrm{Phi}_I,~\mathrm{Phi}_{II},~Z\leftarrow$empty ;\\\\

 \com{\bf Singular problems}\\
 \textbf{else}  $\nofeed\leftarrow I_m$;  $p\leftarrow0;$\\
     \>\textbf{while} \=$\rfeed<m~\&~  p\leftarrow2r+$rank$(\mathrm{Phi},tol );$\\
     \>\> $k\leftarrow k+1;$
    \\

     \>\>\textbf{if} \= $r>0$ \com{There is partial feedback.}\\
        \>\>\> $\rfeed\leftarrow \rfeed+r;$ $[U^T,R,V^T]\leftarrow$svd$(R)$;\\
        \>\>\>\com{Partial feedback}\\
        \>\>\>$R\leftarrow R(1:r,1:r)$;\\
        \>\>\>$\feed \leftarrow R^{-1} * U(1:r,:) * S;$
        $f_{xp}\leftarrow\matrix{c}{f_{xp}\\\feed};$ \\\\
        \>\>\>$\feedtot\leftarrow\matrix{c}{\feedtot\\V(1:r,:) * \nofeed};$\\\\
        \>\>\>$\nofeed\leftarrow \matrix{c}{V * \nofeed}(r+1:m,:);$ \\

     \>\>\com{Vector field, including feedback}\\
     \>\>$G\leftarrow G+ Z(:,1:r) * \feed$; $Z\leftarrow Z(:,r+1:m);$ \\\\
      \>\com{Reduction of 1st class constraints due to feedback.}\\
     \>\textbf{if} \=rank$(\mathrm{Phi}_I,tol)>tol;$\\
        \>\>$F_3\leftarrow \mathrm{Phi}_I(:,2n+1:2n+m)*V^T;$\\
        \>\>$\mathrm{Phi}_I(:,1:2n)\leftarrow  \mathrm{Phi}_I(:,1:2n)+ F_3(:,1:r)*\feed;$\\
        \>\>$F_4\leftarrow[\mathrm{Phi}_I(:,2n+m+1:2n+2m)*V^T](:,r+1:m)$;\\
        \>\>$\mathrm{Phi}_1\leftarrow$\verb+independent_rows+$\big([\mathrm{Phi}_I(:,1:2n),\, F_3(:,r+1:m),\, F_4],tol\big)$;\\
    %\>\textbf{else}\\
    \>\textbf{endif}\\
$\,\cdots/\cdots$\\

\end{tabbing}\hrule\end{minipage}
%\newpage

%%%%%%%%%%%%%%%%%%%%%%%%%%%%%%%%%%%%%%%%%%

\noindent\begin{minipage}[c]{\textwidth}\vspace{0.3cm} \hrule~\\ \textbf{Pseudocode 1:
\emph{HCAPF ``Hamiltonian constraints algorithm with partial feedback on the controls''}}\vspace{0.3cm}\hrule
\begin{tabbing}
\\ $\,\cdots/\cdots$\\
\hspace{0.5cm}\=\\
      \>\com{Reduction of 2nd class constraints due to feedback.}\\
    \> \textbf{if} \=rank$(\mathrm{Phi}_{II},tol)>tol$\\
       \>\> $F_3\leftarrow \mathrm{Phi}_{II}(:,2n+1:2n+m) * V^T;$\\
       \>\>$\mathrm{Phi}_{II}(:,1:2n)\leftarrow \mathrm{Phi}_{II}(:,1:2n)+ F_3(:,1:r) * \feed;$\\
       \>\>$ F_4\leftarrow [\mathrm{Phi}_{II}(:,2n+m+1:2n+2m) * V^T](:,r+1:m);$\\
       \>\> $\mathrm{Phi}_{II}\leftarrow$\verb+independent_rows+$\big([\mathrm{Phi}_{II} (:,1:2n),\, F_3(:,r+1:m),\, F_4],tol\big)$\\
       \>\textbf{endif}\\
       
    \>\com{Iteration of the matrices}\\
    \>\>$R \leftarrow U(r+1:l,:)*  S * Z;$\\
    \>\>$S \leftarrow U(r+1:l,:) * S * G;$\\
\>\textbf{elseif} \com{Iteration of the matrices for $r=0$. There is no feedback.}\\
      \>\>$R\leftarrow S * Z;$\\
      \>\>$S \leftarrow S * G;$\\
\>\=\textbf{endif}\\\\
     \>$r\leftarrow$rank$(R,tol)$; $[l,m]\leftarrow$size$(R)$;
       $p\leftarrow$rank$(\mathrm{Phi},tol );$ \\
    \>\com{New set of constraints}\\
        \> $\mathrm{Phi}\leftarrow$ \verb+independent_rows+ $\bigg(\matrix{ccc}{&\mathrm{Phi}_I&\\ S& R&0_{l\times m}},tol\bigg)$;\\

    \>$\pois\leftarrow$ \verb+poisson_brackets+ $(\mathrm{Phi},n,m,tol)$;  $[v,w]\leftarrow$ \verb+numerical_ker+ $(\pois,tol)$;\\
    \>$\mathrm{Phi}_I\leftarrow v^T * \mathrm{Phi};$  \com{1st class constraints.}\\
    \>$ \mathrm{Phi}_{II}\leftarrow$ \verb+independent_rows+ $\bigg(\matrix{c}{\mathrm{Phi}_{II}\\w^T* \mathrm{Phi}},tol\bigg)$;
    \com{whole set 2nd class constraints.}\\
   \hspace{0.5cm}\= \textbf{endwhile}\\\\
   
    \hspace{0.5cm}\=\com{Computation of the final feedback for the last iteration}\\
    \>\textbf{if}  $r>$\=$tol$\\
        \>\>\com{Partial feedback}\\
      %  \>\>$\rho\leftarrow\rho^{-1}(1:r,1:r)$;\\
        \>\>$\feed\leftarrow R^{-1} * U(1:r,:)*S;$
        $f_{xp}\leftarrow\matrix{c}{f_{xp}\\\feed};$ \\\\
        \>\>$\feedtot\leftarrow\matrix{c}{\feedtot\\V(1:r,:) * \nofeed};$\\
        \>\>$\nofeed\leftarrow \matrix{c}{V * \nofeed}(r+1:m,:);$ \\\\
        
         \hspace{1.3cm}\= \com{Reduction of 1st class constraints due to feedback.}\\
     \>\textbf{if} rank\=$(\mathrm{Phi}_1,tol)>tol;$\\
        \>\>$F_3\leftarrow \mathrm{Phi}_I(:,2n+1:2n+m)*V^T;$\\
        \>\>$\mathrm{Phi}_I(:,1:2n)\leftarrow  \mathrm{Phi}_I(:,1:2n)+ F_3(:,1:r)*\feed_{xp};$\\
        \>\>$F_4\leftarrow[\mathrm{Phi}_I(:,2n+m+1:2n+2m)*V^T](:,r+1:m)$;\\
       \>\>$\mathrm{Phi}_I\leftarrow$\verb+independent_rows+$\big([\mathrm{Phi}_I(:,1:2n),\, F_3(:,r+1:m),\, F_4],tol\big)$;\\
     \>\textbf{endif}\\
 
 $\,\cdots/\cdots$\\

  \end{tabbing}\hrule\end{minipage}
%\newpage

%%%%%%%%%%%%%%%%%%%%%%%%%%%%%%%%%

\noindent\begin{minipage}[c]{\textwidth}\vspace{0.3cm} \hrule~\\ \textbf{Pseudocode 1:
\emph{HCAPF ``Hamiltonian constraints algorithm with partial feedback on the controls''}}\vspace{0.3cm}\hrule
\begin{tabbing}
\\ $\,\cdots/\cdots$\\
          \hspace{0.5cm}\=\\
        \>\com{Reduction of of 2nd class constraints due to feedback.}\\
     \>\textbf{if} rank\=$(\mathrm{Phi}_{II},tol)>tol$\\
       \>\> $F_3\leftarrow \mathrm{Phi}_{II}(:,2n+1:2n+m)*V^T;$\\
       \>\>$\mathrm{Phi}_{II}(:,1:2n)\leftarrow \mathrm{Phi}_{II}(:,1:2n)+ F_3(:,1:r)*\feed;$\\
       \>\>$ F_4\leftarrow [\mathrm{Phi}_{II}(:,2n+m+1:2n+2m)*V^T](:,r+1:m);$\\
       \>\> $\mathrm{Phi}_{II}\leftarrow$\verb+independent_rows+$\big(\mathrm{Phi}_{II}(:,1:2n),\, F_3(:,r+1:m),\, F_4],tol\big)$\\
    \>\textbf{endif}\\
    \hspace{0.5cm}\textbf{endif}\\
    \hspace{0.5cm}$\feed\leftarrow f_{xp}$;\\
\textbf{endif}\\
$m\leftarrow M- \rfeed$\\
%$\pois\leftarrow$\verb+poisson_brackets+$\left(\matrix{c}{F_1\\F_2},n,m,tol\right)$; \com{Poisson bracket matrix of all the constraints}\\
\com{Remove the coisotropic variables.}\\
\textbf{if} \=\rank$(\mathrm{Phi}_I,tol)>tol$;\\
    \>$\mathrm{Phi}_I\leftarrow $\verb+independent_rows+$\big(\mathrm{Phi}_I(:,1:2n+m),tol\big)$; \\
\textbf{endif}\\
\textbf{if} \=\rank$(\mathrm{Phi}_{II},tol)>tol$;\\
    \>$\mathrm{Phi}_{II} \leftarrow $\verb+independent_rows+$\big(\mathrm{Phi}_{II} (:,1:2n+m),tol\big)$; \\
\textbf{endif}\\
\com{Final vector field:  $\dot{x} = A_x x+ A_p p+ B_u u; \dot{p} = Q_x x+ Q_p p+ N_u u$}\\
$A_x\leftarrow G(1:n,1:n)$; $A_p\leftarrow G(1:n,n+1:2n)$;\\
$Q_x\leftarrow G(n+1:2n,1:n)$; $Q_p\leftarrow G(n+1:2n,n+1:2n)$;\\
\textbf{if} $\rfeed$\= $<M$\\
\>$B_u\leftarrow Z(1:M,:)$; $N_u\leftarrow Z(M+1:2M,:)$;\\
\textbf{elseif} $B_u,N_u\leftarrow$ empty;\\
\textbf{endif}\\\\
\textbf{Output} $k,~m,~\mathrm{Phi}_I,~\mathrm{Phi}_{II},~\feed,~\feedtot,~\nofeed,~A_x,~A_p,~Q_x,~Q_p,~B_u,~N_u$;
\end{tabbing}\hrule\end{minipage}

%%%%%%%%%%%%%%%%%%%%%%%%%%%%%%%%%%%%%%%%%%%
%%%%%%%%%%%%%%%%%%%%%%%%%%%%%%%%%%%%%%%%%%%%%%%%%

%\medskip

\noindent\begin{minipage}[c]{\textwidth}\vspace{0.2cm}
\hrule~\\\textbf{Pseudocode 2:}
\verb+independent_rows+ used in HCAPF \vspace{0.2cm}\hrule
\begin{tabbing}
\textbf{procedure} \verb+independent_rows+($\mathrm{\Phi},~tol$)\\\\
\textbf{if}  rank\=$(\mathrm{\Phi},tol)>tol$;\\
   \> $[U,S,V]\leftarrow$ svd$(\mathrm{\Phi})$;\\
   \> $r\leftarrow$rank\=$(\mathrm{\Phi},tol)$;\\
   \> $\mathrm{\Phi}\leftarrow\matrix{c}{U^T \mathrm{\Phi}}(1:r,:)$;\\
\textbf{elseif}  $\mathrm{\Phi}\leftarrow$ empty;\\
\textbf{endif}\\\\
\textbf{output} $$
\end{tabbing}\hrule\end{minipage}

%%%%%%%%%%%%%%%%%%%%%%%%%%%%%%%%%%%%%%%%%%%%%%%%%%

\noindent\begin{minipage}[c]{\textwidth}\vspace{0.2cm} \hrule~\\\textbf{Pseudocode 3:}
\verb+numerical_ker+ used in HCAPF \vspace{0.2cm}\hrule
\begin{tabbing}\\
\textbf{procedure} \verb+numerical_ker+(A,tol)\\\\
$[m,n]\leftarrow$ size($A$); $r\leftarrow$ rank$(A,tol)$; \\
$[U,S,V]\leftarrow$svd$(A)$; ~\=\com{Singular Value Decomposition of $A$}\\
$v\leftarrow V(:,r+1:n)$; ~\> \com{Matrix whose columns span the  kernel of $A$ with tolerance  $tol$}\\
$w\leftarrow V(:,1:r)$; ~\> \com{Matrix whose columns are independent of the columns of $v$,}\\
\textbf{output} $v,~w$
\vspace{0.2cm}
\end{tabbing}\hrule\end{minipage}

%%%%%%%%%%%%%%%%%%%%%%%%%%%%%%%%%%%%%%%%%%%%%%%%%%%

\noindent\begin{minipage}[c]{\textwidth}\vspace{0.2cm}
\hrule~\\\textbf{Pseudocode 4:} \verb+poisson_brackets+ used in HCAPF \vspace{0.2cm}\hrule
\begin{tabbing}\\
\textbf{procedure} \verb+poisson_brackets+($\mathrm{Phi},~n,~m,~tol$)\\
$l\leftarrow$rows$(\mathrm{Phi})$; $\pois\leftarrow 0_{l\times l}$; \\
\textbf{if} rank\=$(\mathrm{Phi},tol)>tol$\\
\>\textbf{for} \=$i=1:~n$\\
\>\>$\pois\leftarrow \pois +\mathrm{Phi}(:,\,i) * \mathrm{Phi}(:,i+n)^T- \mathrm{Phi}(:,\,i+n) * \mathrm{Phi}(:,i)^T$;\\
\>\>\com{Contribution of the state and costate  variables, $(x,p)$.} \\
\>  \textbf{endfor}\\ %part of (x,p)\\
\>\textbf{for} \=$j=1:~m$\\
\>\> $\pois\leftarrow \pois+\mathrm{Phi}(:,j+2n) * \mathrm{Phi}(:,j+2n+m)^T- \mathrm{Phi}(:,j+2n+m) * \mathrm{Phi}(:,\,j+2n)^T$\\
\>\>\com{Contribution of the control and coisotropic  variables, $(u,v)$.}\\
\>\textbf{endfor}\\
\>$\pois\leftarrow\dfrac{\pois-\pois^T}{2}$\\
\textbf{elseif}\\
\>  $\pois\leftarrow[~]$\\
\textbf{endif}\\
\textbf{output}~ $\pois$
\end{tabbing}\hrule
\end{minipage}

\bigskip

Above we show the pseudocodes for this algorithm as well as the subroutines used in it: \verb+independent_rows+, \verb+numerical_ker+, and \verb+poisson_brackets+. The function \verb+independent_rows+ remove the dependent 
rows of a matrix $\Phi$.  The function \verb+numerical_ker+ computes a matrix $v$ 
whose columns span the kernel of $A$ and a matrix $w$ whose columns are independent of the columns 
of $v$. Finally,  \verb+poisson_brackets+ generates the matrix $\pois=(\pois_{ij})$, 
where each element $\pois_{ij}$ is the
Poisson bracket  of the $i$-th row  with the  $j$-th row of the matrix $\Phi$. The
dimensions $n$ and $m$ of the state  and control variables, $x$ and $u$, are needed to compute the
Poisson bracket.
%%%%%%%%%%%%%%%%%%%%%%%%%%%%%%%%%%%%%%%%%%%%%%%%%%%

\section{Examples and numerical experiments}\label{experiments}
We  discuss here some numerical experiments that exhibit the stability and consistency of the numerical algorithm discussed above. The microprocessor used for the
numerical computations was  Intel(R), Core(TM) i7-2640M, CPU 2.80 GHz, 2.80 GHz, 4.00 GB RAM, 64 bits, and the program was MATLAB 7.12.0.

We describe two sets of experiments with small and large recursive index respectively.  We test the algorithm against a class of non--trivial problems general enough for the purposes of exhibiting its numerical stability but that can be solved exactly.

In the small index problems ($\nu =2$), we show  that the algorithm is stable with respect to the
tolerance used to compute the numerical rank, $tol$, and with respect to perturbations of the data,
$\delta$. We will also discuss the dependence with the size of the matrices, $n$.

In the large index case, we analyze a problem of index  $n-1$, where the algorithm behaves properly
with respect to the number of steps, both regarding the tolerance, $tol$, and the perturbation of
the data, $\delta$.

%%%%%%%%%%%%%%%%%%%%%%%%%%%%%%%%%%%%%%%%%%%
%%%%%%%%%%%%%%%%%%%%%%%%%%%%%%%%%%%%%%%%%%%

\subsection{Small index problems. Small matrices}

Consider a positive semidefinite symmetric $n\times n$ matrix $R$ of rank $r$, thus there will
exist an orthogonal matrix $U$ such that
\eq{\label{svdR} R = U^T R' U,}%
where all elements of $R'$ vanish except $R'_{11},\ldots,R'_{rr} > 0$. State and control spaces are both $\R^n$, and the total space $(x,p,u)$ is $\R^{3n}$. The matrix $A$ is the identity and $B$ is an orthogonal matrix, $B^T B = I_n$.  Finally, the objective functional is constructed by using a generic
diagonal matrix $Q=D$, a matrix $N$ of the form $N = BV,$ where $V$ is a symmetric matrix constructed as 
$V = \frac12B^T\cdot D_l\cdot B$, where $D_l$ is the truncated diagonal matrix of $Q$ with only the  $l$ first diagonal entries non--zero, and  the
matrix $R$ described above. The primary constraints, taking into account that $B^TN - N^TB = 0$ (without feedback) are:
$$ \phi^{(1)} = -N^T x+ B^Tp  - Ru= 0, $$
an we will obtain the partial feedback of $r$ controls.  The reduced constraint will be:
$$\widetilde{\phi}^{(1)}(x,p,u) =  U^1(-N^TA + B^TQ)x - U^1B^TA^Tp = 0 ,$$
and computing  the derivative we get: 
\begin{eqnarray}
\dot{\widetilde{\phi}}^{(1)} &=&
U^1(-N^TA^2 +  B^TQA - B^TA^TQ)x + U^1B^T(A^T)^2 p + \\ \nonumber
&+&U^1(B^TQB - N^TAB - B^TA^TN )u = 0 .
\end{eqnarray}

The matrix  $B^TQB - N^TAB - B^TA^TN$ will be invertible for generic $A$, $Q$, $B$ and
$V$, but for our problem the matrix $ R^{(2)} = U^1(B^TQB - 2V)$ has rank $m-l$, so we get partial feedback of $m-l$ controls and the algorithm stops here since the rest is proportional to the previous constraints.

The numerical experiment  of this problem will consist in applying the algorithm to a
collection of matrices built up as a random perturbation of size $\delta$ of the matrices $A$, $B$, $Q$, $N$: 
$\tilde{A}=A+\delta A,$ $\|\delta A\|<\delta,$ $\tilde{B}=B+\delta B,$ $\|\delta B\|<\delta,$
$\tilde{Q}=Q+\delta Q,$ $\|\delta Q\|<\delta$ and $\tilde{N}=N+\delta N,$ $\|\delta N\|<\delta$. It
is computed for $n=100,~200$, with $r=n-20$ and $l=5$. 

We analyze the number of steps before the algorithm stabilizes and compute the angle $\alpha$, between the final constraint submanifold of the perturbed problem and
the exact one, this is the error introduced in the problem by the perturbations $\delta A$, $\delta
B$, $\delta Q$ and $\delta N$.   Notice that the original matrix $R$ does not affect higher order constraints, hence
its numerical influence restricts just to launching the algorithm until they are of order of the tolerance, $tol=10^{-6}$. 

Table \ref{t1} shows the exact and computed steps of the algorithm, the leftover controls, $m= n-(r+l)= n-(n-20+5)=15$
and the rank $rp=10$ of the Poisson matrix, that gives us the number of second class constraints.  We compare it with the
perturbed problem  (number of steps, $ m_1$ and $rp_1$). 

Finally, we compute the error ($\alpha$), measured as the angle
between both set of constraints. This is only computable when both set of constraints have the same number of columns, so it fails when the leftover controls are not the same. 
The results  in Table \ref{t1} show that the algorithm works well up to perturbations of order of $\delta=10^{-6}=tol$, even though the nature of 
first/second class constraints fails at $\delta=10^{-7}$ and so does the number of steps for $n=200$. The
least squares approximation of $\log_{10} \alpha$ versus $\log_{10} \delta$ gives two lines of slopes  $0.9765$ (for $n=100$)and $0.9803$ (for $n=200$), which is consistent with $\alpha=O(\delta).$ (See Figure \ref{figure1}).

%
% Table calculated with the archive ex1.m
%
\begin{center}
\begin{table}
\caption{First experiment  (small index, small matrices). $tol=10^{-6}$ }\label{t1}
\begin{tabular}{ccccccccc}\hline
$n$&$\delta$&\# exact steps&\# steps&$m$&$m_1$&$rp$&$rp_1$&$\alpha$/ error\\\hline
100&1e-013& 3& 3&15&15&10&10&0.0000000000561788\\
100&1e-012& 3& 3&15&15&10&10&0.0000000005405419\\
100&1e-011& 3& 3&15&15&10&10&0.0000000054350352\\
100&1e-010& 3& 3&15&15&10&10&0.0000000548428769\\
100&1e-009& 3& 3&15&15&10&10&0.0000004943661747\\
100&1e-008& 3& 3&15&15&10&10&0.0000049043295974\\
100&1e-007& 3& 3&15&15&10&12&0.0000373718753680\\
100&1e-006& 3& 3&15& 0&10&10&not computable\\
100&1e-005& 3& 2&15& 0&10& 2&not computable\\
\hline
200&1e-013& 3& 3&15&15&10&10&0.0000000000108757\\
200&1e-012& 3& 3&15&15&10&10&0.0000000001127705\\
200&1e-011& 3& 3&15&15&10&10&0.0000000010570457\\
200&1e-010& 3& 3&15&15&10&10&0.0000000113724130\\
200&1e-009& 3& 3&15&15&10&10&0.0000001093589043\\
200&1e-008& 3& 3&15&15&10&10&0.0000011856677989\\
200&1e-007& 3& 7&15&15&10&22&0.0000068162846885\\
200&1e-006& 3& 3&15& 0&10&12&not computable\\
200&1e-005& 3& 2&15& 0&10& 0&not computable\\
\hline\end{tabular}
\end{table}
\end{center}

In this experiment it was also computed the improvement obtained by performing partial feedback on the controls at each step versus the plain constraints algorithm.  When matrices are large enough the improvement in CPU time is in the range $[58.1\%,118.4\%]$ for $n=100$ and 
$[76.1\%,200.4\%]$ for $n=200$.   

In Table \ref{t2} we have selected a fixed value for the perturbation $\delta$ (= $10^{-12}$). We  can observe that the algorithm is almost insensitive to the size of the matrices, $n$.  Again, the improvement in computation time due to partial feedback is in the range of $[84.9\%,171.6\%]$.

\begin{figure}[h!]
\begin{center}
\includegraphics[width=8cm]{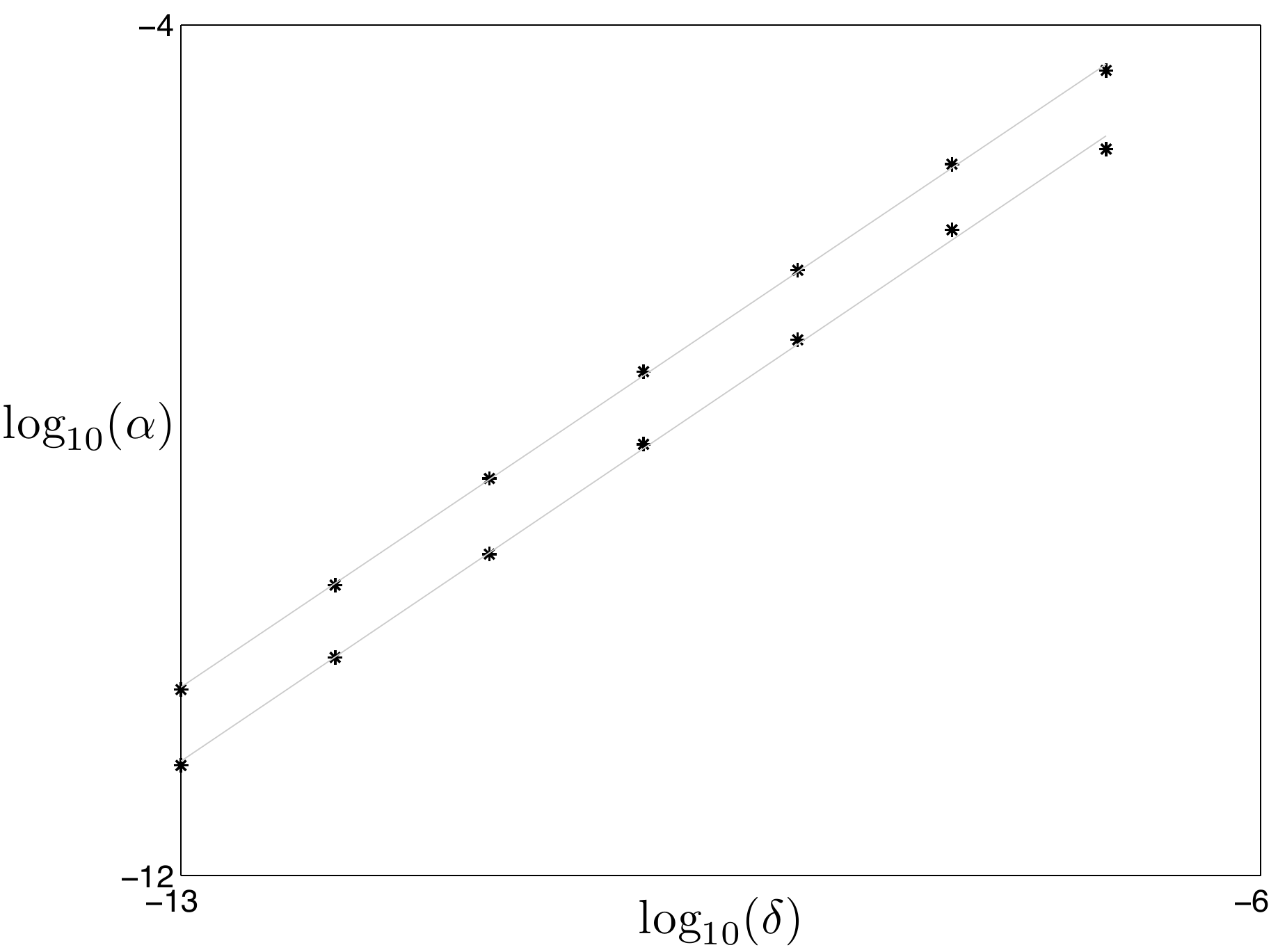}
\caption{Error for the first experiment  (small index, small matrices). $n=100,~200$, $tol=10^{-6}$ }\label{figure1}
\end{center}
\end{figure}

% Table calculated with the archive ex2.m

\begin{center}
\begin{table}
\caption{First experiment  (small index, small matrices). $\delta=10^{-12},$
$tol=10^{-6}$}\label{t2}
\begin{tabular}{cccccccc}\hline
$n$&\# exact steps&\# steps&$m$&$m_1$&$rp$&$rp_1$&$\alpha$/ error\\\hline
100& 3& 3&15&15&10&10&0.0000000000464235\\
125& 3& 3&15&15&10&10&0.0000000004122544\\
150& 3& 3&15&15&10&10&0.0000000000829511\\
175& 3& 3&15&15&10&10&0.0000000011055558\\
200& 3& 3&15&15&10&10&0.0000000001268241\\
\hline\end{tabular}
\end{table}
\end{center}

%%%%%%%%%%%%%%%%%%%%%%%%%%%%%%%%%%%%%%%
%%%%%%%%%%%%%%%%%%%%%%%%%%%%%%%%%%%%%%%

\subsection{Small index problems. Large matrices}
Let us consider now the linear--quadratic problem given by $Q=A=I_n \in\R^{n\times n},$
$B^T=(1,\ldots,1)\in\R^{n\times 1},~N^T=(0,\ldots,0)\in\R^{n\times 1},~R=0;$ so the constraints
matrix will be (as shown in \cite{De09})
\begin{equation}
\P = \matrix{ccc|rcr|c|c}{0 & \cdots & 0 & 0 & \cdots & 0 & 0 & 1 \\0 & \cdots & 0 & 1 & \cdots &1 & 0 & 0 \\ 1 & \cdots & 1 & -1 & \cdots & -1 & 0 & 0 \\
0 & \cdots & 0 & 1 & \cdots & 1 & n & 0 },
\end{equation}
where $n$ is the dimension of the matrices $A$ and $Q$.
Thus, in the third row we obtain  optimal feedback and the final constraint submanifold is given by
the following equations: $x_1+\cdots +x_n=p_1+\cdots+ p_n=u=0$.

We apply the numerical algorithm  for the previous matrices for $n=3000$, tolerance equal to
$10^{-6}$ and we compare the solution obtained with the perturbed  matrices: $\tilde{A}=A+\delta
A,$ $\|\delta A\|<\delta,$ $\tilde{N}=N+\delta N,$ $\|\delta N\|<\delta$, $\tilde{B}=B+\delta B,$
$\|\delta B\|<\delta,$ and  $\tilde{Q}=Q+\delta Q,$ $\|\delta Q\|<\delta$, where
$\delta=10^{-13}\to 10^{-5}$. 

Again, as the final constraint submanifold of the original problem
and the perturbed one must be the same, we measure the angle between them (the error),  we show it
in Table \ref{t3}. 
The slope of the least squares approximation of $\log_{10}(\alpha)$ versus
$\log_{10}(\delta)$ is $0.9988$ (see Figure \ref{figure2}). 
Again, consistently with the previous results, the data shows that
$\alpha=O(\delta)$, and fails when the perturbation is of order of the tolerance.

% Table calculated with the archive ex3.m
\begin{center}
\begin{table}[h!]
\caption{Second experiment. $n=3000$.  $tol=10^{-6}$}\label{t3}
\begin{tabular}{cccccccc}\hline
$\delta$&\# exact steps&\# steps&$m$&$m_1$&$rp$&$rp_1$&$\alpha$/ error\\\hline
1e-013& 3& 3& 0& 0& 2& 2&0.0000000000022437\\
1e-012& 3& 3& 0& 0& 2& 2&0.0000000000223028\\
1e-011& 3& 3& 0& 0& 2& 2&0.0000000002226745\\
1e-010& 3& 3& 0& 0& 2& 2&0.0000000022331076\\
1e-009& 3& 3& 0& 0& 2& 2&0.0000000223615340\\
1e-008& 3& 3& 0& 0& 2& 2&0.0000002228797000\\
1e-007& 3& 3& 0& 0& 2& 2&0.0000022084665700\\
1e-006& 3& 3& 0& 0& 2& 2&0.0000218159062779\\
1e-005& 3& 0& 0& 0& 2& 0&not computable\\
\hline\end{tabular}
\end{table}
\end{center}

\begin{figure}[h!]
\begin{center}
\includegraphics[width=8cm]{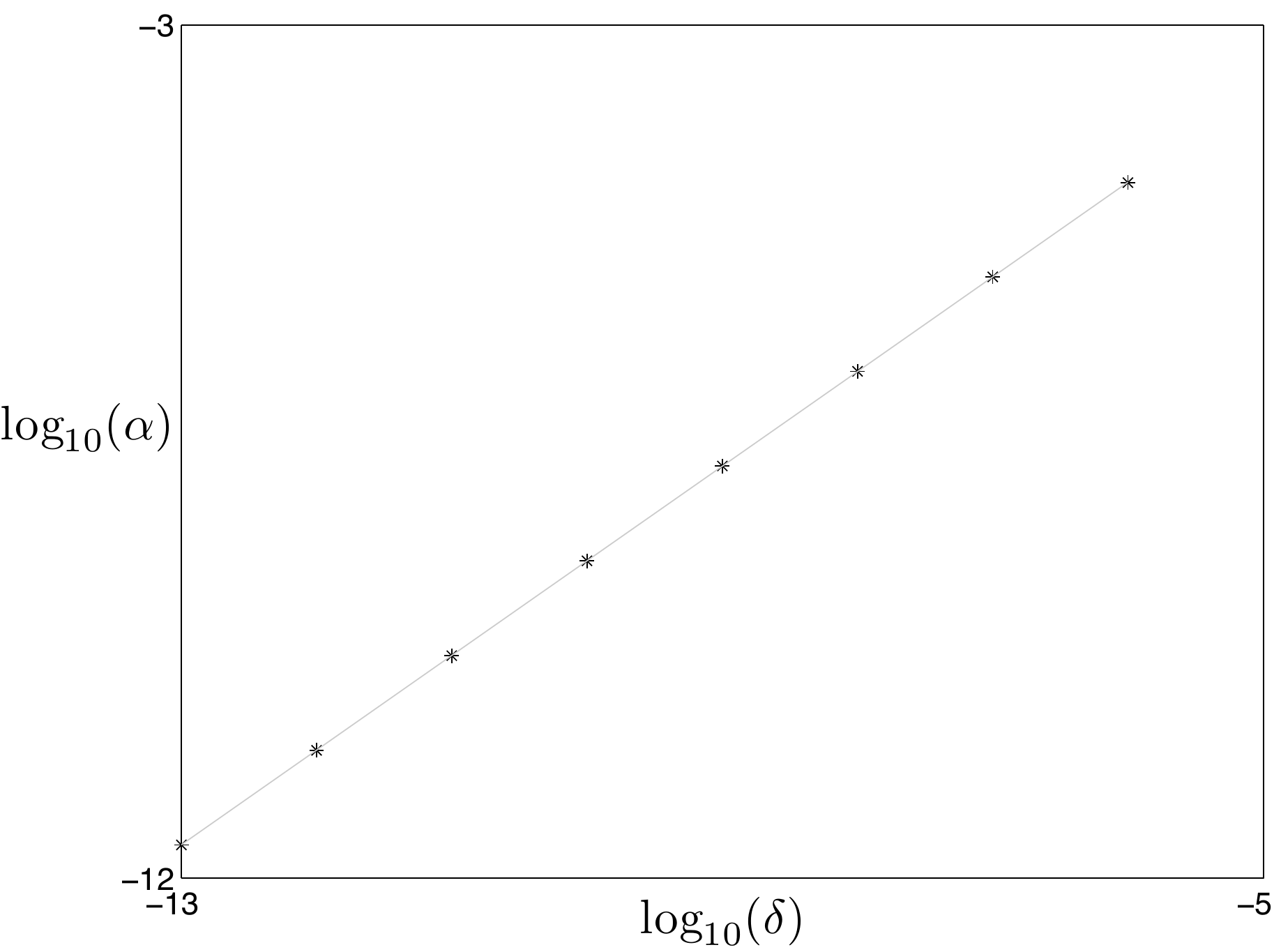}
\caption{Error for the second experiment (small index, large matrices).
$n=3000$.  $tol=10^{-6}$}\label{figure2}
\end{center}
\end{figure}

%%%%%%%%%%%%%%%%%%%%%%%%%%%%%%%%%%%%%%%%
%%%%%%%%%%%%%%%%%%%%%%%%%%%%%%%%%%%%%%%%

\subsection{Large index problems}
Consider the following problem:
\eq{ A\in\R^{n\times n}, ~Q=A+A^T, ~B\in\R^{n\times 1}, ~N=B, R=0.}
As shown in \cite{De09}, $\rho^{(1)} = \rho^{(2)}=\cdots \rho^{(k+1)}=0$,
since all these matrices are always zero,  there will not be optimal feedback and the constraints matrix will be:
\eqi{\P &=& \matrix{ccc|c|c}{0 & 0 & & 0& 1\\ -B^T & B^T & & 0& 0\\ B^TA^T & -B^TA^T & & 0 & 0\\-B^T(A^T)^2 & B^T(A^T)^2 & &0& 0\\
\vdots & \vdots & & \vdots& \vdots\\ -(-1)^kB^T(A^T)^k & (-1)^kB^T(A^T)^k & & 0& 0}.}%
The algorithm will stop   when at a
given step we  obtain a linear combination of the previous rows.
We apply the algorithm to a problem where the pair $(A,B)$ is such that  \mbox{$A\in\R^{n\times
n}$} is a nilpotent matrix of index $n$, i.e.,  $A^{n-1}\neq 0$, $A^n=0$, and $B\notin~\ker(A^l)$,
\mbox{$l=1,\ldots,n-1$,} i.e., $B^T=(1,\ldots,1)\in\R^{1\times n}$.  Thus the index of the
algorithm is $n-1$ ($n$ steps). For $n=20$, $\delta=10^{-13}\to 10^{-5}$, and tolerance equal to
$10^{-6}$, we perturb the matrices as: $\tilde{A}=A+\delta A,$ $\|\delta A\|<\delta,$
$\tilde{N}=N+\delta N,$ $\|\delta N\|<\delta$, $\tilde{B}=B+\delta B,$ $\|\delta B\|<\delta$ and
$\tilde{Q}=\tilde{A}+\tilde{A^T},$ $\|\delta Q\|<\delta$. The perturbation of $\tilde{Q}$   is done
to maintain the structure of the problem.  The results are shown in Table \ref{t4} and Figure
\ref{f4}. The slope of the least squares fitting of the resulting data is $1.0115$, that is,
$\alpha=O(\delta)$ (see Figure \ref{figure4}).

%%Table calculated with ex3.m
\begin{center}
\begin{table}[h!]\caption{Third experiment (large index).  $tol=10^{-6}$}
\label{t4}
\begin{tabular}{ccccccccc}\hline
$n$&$\delta$&\# exact steps&\# steps&$m$&$m_1$&$rp$&$rp_1$&$\alpha$/ error\\\hline
20&1e-013&20&20& 1& 1& 0& 0&0.0000000000005502\\
20&1e-012&20&20& 1& 1& 0& 0&0.0000000000053287\\
20&1e-011&20&20& 1& 1& 0& 0&0.0000000000891883\\
20&1e-010&20&20& 1& 1& 0& 0&0.0000000003152048\\
20&1e-009&20&20& 1& 1& 0& 0&0.0000000039064143\\
20&1e-008&20&20& 1& 1& 0& 0&0.0000000375324515\\
20&1e-007&20&20& 1& 1& 0& 2&0.0000011726145723\\
20&1e-006&20& 3& 1& 0& 0& 2& not computable\\
20&1e-005&20& 0& 1& 0& 0& 0& not computable\\
\hline\end{tabular}
\end{table}
\end{center}

%\newpage

\begin{center}
\begin{figure}[h]\label{f4}
\includegraphics[width=8cm]{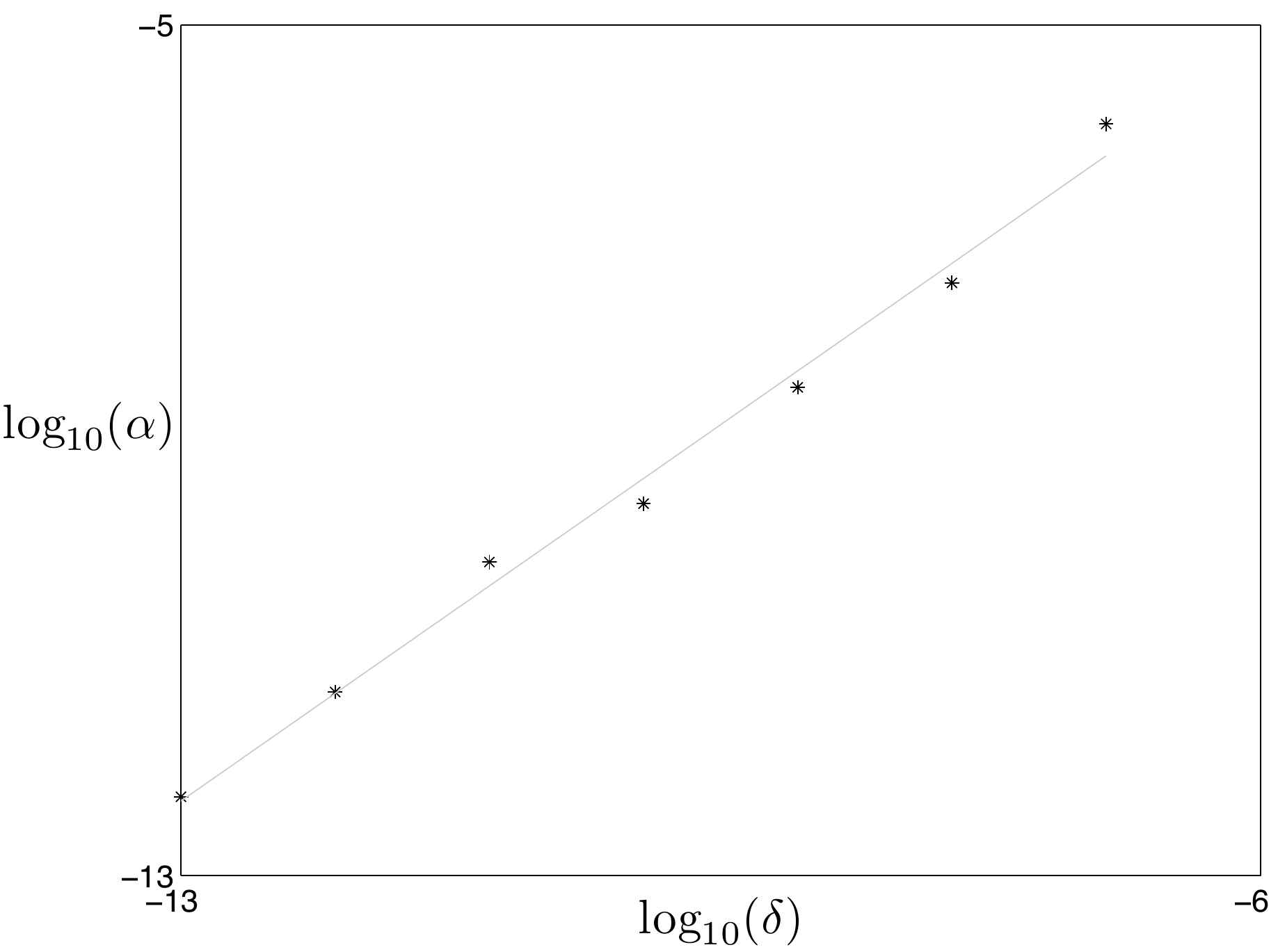}
\caption{Error for the third experiment (large index). $n=20$,  $tol=10^{-6}$}\label{figure4}
\end{figure}
\end{center}

%%%%%%%%%%%%%%%%%%%%%%%%%%%%%%%%%%%%%%%%%%%%%%%%
%%%%%%%%%%%%%%%%%%%%%%%%%%%%%%%%%%%%%%%%%%%%%%%%

\end{document}